\documentclass[10pt,reqno]{amsart}

\usepackage{a4wide}
\usepackage{dsfont}
\usepackage{graphicx}
\usepackage{amssymb}
\usepackage[all]{xy}

\newtheorem{proposition}{Proposition}[section]
\newtheorem{theorem}[proposition]{Theorem}
\newtheorem{lemma}[proposition]{Lemma}

\theoremstyle{remark}
\newtheorem{definition}[proposition]{Definition}
\newtheorem{remark}[proposition]{Remark}

\newcommand{\cst}{\ifmmode\mathrm{C}^*\else{$\mathrm{C}^*$}\fi}
\newcommand{\st}{\;\vline\;}
\newcommand{\St}{{\bigl.\st\bigr.}}
\newcommand{\CC}{\mathbb{C}}
\newcommand{\RR}{\mathbb{R}}
\newcommand{\ZZ}{\mathbb{Z}}
\newcommand{\TT}{\mathbb{T}}

\newcommand{\tens}{\otimes}
\newcommand{\id}{\mathrm{id}}
\newcommand{\comp}{\!\circ\!}
\newcommand{\I}{\mathds{1}}
\newcommand{\hh}[1]{\widehat{#1}}
\newcommand{\bb}[1]{\overline{#1}}

\newcounter{c}

\newcommand{\GG}{\mathbb{G}}
\newcommand{\HH}{\mathbb{H}}
\newcommand{\cA}{\mathcal{A}}
\newcommand{\cB}{\mathcal{B}}
\newcommand{\II}{\mathds{I}}
\newcommand{\vtens}{\bar{\tens}}
\newcommand{\nc}{{\text{---}\scriptscriptstyle{\|\cdot\|}}}
\newcommand{\uu}{{\scriptscriptstyle\mathrm{u}}}
\newcommand{\tp}{\xymatrix{*+<.7ex>[o][F-]{\scriptstyle\top}}}
\newcommand{\cI}{\mathcal{I}}
\newcommand{\ww}{\mathrm{W}}
\newcommand{\WW}{{\mathds{V}\!\!\text{\reflectbox{$\mathds{V}$}}}}
\newcommand{\Ww}{\mathds{W}}
\newcommand{\wW}{\text{\reflectbox{$\Ww$}}\:\!} 
\newcommand{\pir}{\pi_r}

\newcommand{\sA}{\mathsf{A}}
\newcommand{\sB}{\mathsf{B}}
\newcommand{\sC}{\mathsf{C}}
\newcommand{\sM}{\mathsf{M}}
\newcommand{\sN}{\mathsf{N}}

\newcommand{\sH}{\mathsf{H}}
\newcommand{\sK}{\mathsf{K}}

\DeclareMathOperator{\C}{C}
\DeclareMathOperator{\D}{D}
\DeclareMathOperator{\M}{M}
\DeclareMathOperator{\Mor}{Mor}
\DeclareMathOperator{\cc}{c}
\DeclareMathOperator{\B}{B}
\DeclareMathOperator{\K}{\mathcal{K}}

\DeclareMathOperator{\vN}{vN}

\numberwithin{equation}{section}

\begin{document}

\author{Matthew Daws}
\address{School of Mathematics, University of Leeds, Leeds LS2 9JT, United Kingdom}
\email{matt.daws@cantab.net}

\author{Pawe{\l} Kasprzak}
\address{Department of Mathematical Methods in Physics, Faculty of Physics, University of Warsaw, Poland
and Institute of Mathematics of the Polish Academy of Sciences,
ul.~\'Sniadeckich 8, 00--956 Warszawa, Poland}  \email{pawel.kasprzak@fuw.edu.pl}

\author{Adam Skalski}
\address{Institute of Mathematics of the Polish Academy of Sciences,
ul.~\'Sniadeckich 8, 00--956 Warszawa, Poland}
\email{a.skalski@impan.pl}

\author{Piotr M.~So{\l}tan} \address{Department of Mathematical Methods in Physics, Faculty of Physics, University of Warsaw, Poland}
\email{piotr.soltan@fuw.edu.pl}

\title{Closed quantum subgroups of locally compact quantum groups}


\dedicatory{Dedicated to Leonid Vainerman on the occasion of his 65th birthday}

\begin{abstract}
We investigate the fundamental concept of a closed quantum subgroup of a locally compact quantum group. Two definitions --- one due to S.~Vaes and one due to S.L.~Woronowicz --- are analyzed and relations between them discussed. Among many reformulations we prove that the former definition can be phrased in terms of quasi-equivalence of representations of quantum groups while the latter can be related to an old definition of Podle\'s from the theory of compact quantum groups. The cases of classical groups, duals of classical groups, compact and discrete quantum groups are singled out and equivalence of the two definitions is proved in the relevant context. A deep relationship with the quantum group generalization of Herz restriction theorem from classical harmonic analysis is also established, in particular, in the course of our analysis we give a new proof of Herz restriction theorem.
\end{abstract}

\date{\today}

\maketitle

\section{Introduction}

In this paper we study the notion of a closed quantum subgroup of a locally compact quantum group. The theory of quantum groups phrased in operator algebra language is already well established as a rapidly developing field on the border between noncommutative geometry and abstract harmonic analysis. Nevertheless, the fundamental notion of a closed (quantum) subgroup has not received enough attention so far. There have been several ``working definitions'' of such an object, but most efforts were directed toward developing other aspects of the theory. The first to look at quantum subgroups of (compact) quantum groups was P.~Podle\'s (\cite{podlesPhD,podles}, see also a later discussion in \cite{Pin}). His view was motivated by the straightforward noncommutative generalization of the inclusion homomorphism from the subgroup to the group and required the existence of a surjective $*$-homomorphism between the algebras of continuous functions on respective quantum groups.
This point of view, however, has many disadvantages and drastically limits the number of subgroups (e.g.~many quantum groups do not have the trivial subgroup in this sense). Soon it was realized that in the context of compact quantum groups one should rather require the existence of a surjective $*$-homomorphism between the \emph{universal} versions of algebras of continuous functions on respective quantum groups. This approach, adopted for example in \cite{BB} and \cite{nchom}, avoids the problems mentioned above and also enables a purely algebraic reformulation in terms of the underlying Hopf $*$-algebras. It was not clear, however, whether it would lead to a satisfactory notion for arbitrary \emph{locally compact} quantum groups.

In 2005 in \cite{Vaes} S.~Vaes proposed another definition of a closed quantum subgroup of a locally compact quantum group, phrased in the language of von Neumann algebras. This definition was used in the same paper to develop the full force of the theory of induced representations and homogeneous spaces for quantum groups. Earlier another definition of a closed subgroup of a locally compact quantum group was proposed in \cite[Definition 2.9]{VV} by Vaes and Vainerman. We show that the definition of Vaes and that given by Vaes and Vainerman are equivalent. It should be stressed that the argument needed to show that the definitions of a closed quantum subgroup proposed in \cite{Vaes,VV} give the standard notion of a closed subgroup in classical case is quite subtle. It can be formulated as saying that an inclusion of a closed subgroup $H$ into a locally compact group $G$ induces a normal inclusion of respective group von Neumann algebras $\vN(H)\hookrightarrow\vN(G)$ and is
equivalent to the fact that the restriction to $H$ of regular representation of $G$ is quasi-equivalent to the regular representation of $H$. This is, in turn, equivalent to the conclusion of the Herz restriction theorem which says that the map of Fourier algebras associated to $H\hookrightarrow{G}$ is a surjective contraction (\cite{herz}, cf.~also Section \ref{commut}). All this has been known to the experts for a long time (cf.~\cite{h2,TT, DelDer, Der,VV}); a detailed proof can be found  in the 2008 thesis of C.~Zwarich (\cite{zwarich}).

The definition given in \cite{Vaes} is very well adapted to the problems studied in that paper, but it was not clear whether it is optimal in other contexts and how it relates to the notion studied earlier for compact quantum groups. As mentioned above it is also relatively difficult to see that it actually generalizes the classical notion of a closed subgroup. Yet another possible definition, related to the recently introduced notion of morphisms between quantum groups (\cite{MRW}), was suggested to us by S.L.~Woronowicz. Woronowicz's definition is phrased entirely in the language of \cst-algebras and it is notably easier to see that it generalizes the ordinary notion of a closed subgroup of a locally compact group (see Section \ref{commut}). The main focus of this paper is on understanding the relations between the definitions of a closed quantum subgroup of a locally compact quantum group given by Vaes and Woronowicz and providing their equivalent reformulations.

The definition of Woronowicz is deeply connected with the notion of a \cst-algebra generated by a quantum family of multipliers (which we analyze in Subsection \ref{morphS}) and turns out to be equivalent to the reformulation of the original idea of Podle\'s, i.e.~corresponds to the existence of a surjective $*$-homomorphism between the universal versions of the algebras of continuous, vanishing at infinity, functions on respective locally compact quantum groups. On the other hand the definition of Vaes can be rephrased in a simplified way (still in the von Neumann algebraic language) and turns out to be intimately connected with the notion of quasi-equivalence of representations of quantum groups (\cite{WorSol}, cf.~Theorem \ref{Vaes-equiv}). Moreover, we show that this definition of a closed subgroup of a quantum group is strongly tied to the generalization to quantum groups of the Herz restriction theorem (cf.~Remark \ref{Vf}).

We show that the definition of Vaes is stronger than the definition of Woronowicz, in the sense that if $\HH$ and $\GG$ are locally compact quantum groups and $\HH$  is a closed quantum subgroup of  $\GG$ in the sense of Vaes, then it is also a closed quantum subgroup of  $\GG$ in the sense of Woronowicz. Further we prove that they are equivalent in all special cases one usually considers: classical groups (both definitions describe the standard notion of a closed subgroup), duals of classical groups (both definitions describe a group epimorphism in the opposite direction), compact quantum groups and discrete quantum groups (Sections \ref{commut}, \ref{cocommut} and \ref{cptdiscr}). In particular this opens the way to finding all compact quantum subgroups of a given locally compact quantum group $\GG$ via the theory of idempotent states (as studied for example in \cite{pa}) since each compact quantum subgroup of $\GG$ gives rise to a state on the algebra of functions
on $\GG$, which is idempotent with respect to the convolution product, and such states can be sometimes computed directly using Fourier transform methods. In the context of compact quantum groups this strategy was employed in \cite{FST} to re-establish the list of all quantum subgroups of $\mathrm{SU}_q(2)$, originally found by Podle\'s in \cite{podles}.

In the course of our investigation we make crucial use of the quantum group versions of the Fourier and Fourier-Stieltjes algebras (cf.~\cite[Section 8]{disser}). It is worth noting that our work produces a new proof of the classical Herz restriction theorem (see Section \ref{commut}). In the group-dual case we use the results of M.~Ilie and R.~Stokke on weak$^*$-continuous maps of Fourier-Stieltjes algebras (\cite{is}) which we are also able to generalize (to some extent) to the quantum group setting (Proposition \ref{ISprop}). This exemplifies the connections of our article with recent extensions of noncommutative harmonic analysis to the context of locally compact quantum groups (see for example \cite{hnr} and references therein).

Finally let us note that the differences between the definitions of a closed quantum subgroup according to Vaes and Woronowicz bear a striking similarity to the interplay between the Kustermans-Vaes definition of a locally compact quantum group (formulated in \cite{KV}) and the definition of a quantum group used in \cite{WorSol,MRW} and based on the theory of manageable and modular multiplicative unitaries (\cite{mu,modmu}). Again, the former definition is stronger and in all examples one finds that the two approaches are equivalent. Moreover, in special cases of classical groups, duals of classical groups, compact and discrete quantum groups we have results on existence of Haar measures, so the Kustermans-Vaes approach is equivalent with the one used by So{\l}tan-Woronowicz.

At the present stage of research in the theory of quantum groups it is very difficult to predict whether the definitions of a closed quantum subgroup given by Vaes and Woronowicz are equivalent. We conjecture that in the full generality they are different. However, it seems very likely that in large classes of well-behaved locally compact quantum groups, e.g.~the regular or even semi-regular ones, the two definitions will turn out to be equivalent.

Let us give now a brief description of the paper. In the remainder of this section we collect necessary preliminaries from the theory of \cst-algebras (Subsection \ref{morphS}), locally compact quantum groups (Subsection \ref{QGsub}) and homomorphisms of quantum groups as defined in \cite{MRW} (Subsection \ref{morphism}). Section \ref{Reps} focuses on the theory of representations of quantum groups and the notion of quasi-equivalence of such representations. We also relate this notion to the problem of generation of \cst-algebras by quantum families of multipliers, which later turns out to be crucial for understanding the interplay between the definitions of closed quantum subgroups given by Vaes and Woronowicz. These are introduced in Section \ref{Subgroups} with the relations between them unraveled. We provide several equivalent reformulations of either definition and show that the former implies the latter (in the sense described above). We also give sufficient
conditions for the two definitions to be equivalent. Section \ref{commut} is devoted to the study of both definitions of a closed quantum subgroup in the special case of classical groups. We prove there in detail that both are equivalent to the standard definition of a closed subgroup and discuss the direct connection between the definition of Vaes and the Herz restriction theorem. Then in Section \ref{cocommut} we conduct a similar investigation for the case of duals of classical groups. In this case also the definitions of Vaes and Woronowicz agree. Finally in Section \ref{cptdiscr} we show that the two definitions are equivalent for compact and discrete quantum groups (more precisely a compact quantum group $\HH$ is a closed subgroup of a locally compact quantum group $\GG$ in the sense of Vaes if and only if it is a closed subgroup of $\GG$ in the sense of Woronowicz, and a similar result holds for subgroups of discrete quantum groups).

\subsection{\cst-algebras and morphisms}\label{morphS}

Throughout the paper we will use the language of the theory of \cst-algebras as introduced in \cite{pseu,unbo,gen,lance}. In particular for \cst-algebras $\sA$ and $\sB$ a \emph{morphism} from $\sA$ to $\sB$ is a $*$-homomorphism $\Phi$ from $\sA$ into the multiplier algebra $\M(\sB)$ of $\sB$ which is non-degenerate, i.e.~the set $\Phi(\sA)\sB$ of linear combinations of products of the form $\Phi(a)b$ ($a\in\sA$, $b\in\sB$) is dense in $\sB$ (by the Cohen factorization theorem this is equivalent to the condition that $\Phi(\sA)\sB=\sB$). The set of all morphisms from $\sA$ to $\sB$ will be denoted by $\Mor(\sA,\sB)$. The non-degeneracy of morphisms ensures that each $\Phi\in\Mor(\sA,\sB)$ extends uniquely to a unital $*$-homomorphism $\M(\sA)\to\M(\sB)$ which we will sometimes denote by $\bar{\Phi}$. This also defines the operation of composition of morphisms (see \cite{pseu,gen,lance}).
For a Hilbert space $\sH$ the \cst-algebra of compact operators on $\sH$ will be denoted by $\K(\sH)$. Any \cst-algebra $\sA$ acting on $\sH$ (written $A\subset\B(\sH)$) will act \emph{non-degenerately}, so that the identity map $\id_{\sA}\colon\sA\to\sA$ is a morphism from $\sA$ to $\K(\sH)$. More generally a \emph{representation} of $\sA$ on $\sH$ is by definition an element of $\Mor\bigl(\sA,\K(\sH)\bigr)$.

The notion of a morphism of \cst-algebras generalizes that of a continuous map between locally compact Hausdorff spaces. We have the following well known result:

\begin{theorem}\label{std}
Let $X$ and $Y$ be locally compact Hausdorff spaces and let $\sB=\C_0(X)$ and $\sA=\C_0(Y)$. Then
\begin{enumerate}
\item any continuous $\phi\colon{X}\to{Y}$ defines a morphism $\Phi\in\Mor(\sA,\sB)$ via
\begin{equation}\label{Phif}
\Phi(f)=f\comp\phi,\qquad(f\in\sA);
\end{equation}
\item for any $\Phi\in\Mor(\sA,\sB)$ there exists a continuous $\phi\colon{X}\to{Y}$ such that \eqref{Phif} holds.
\setcounter{c}{\value{enumi}}
\end{enumerate}
Fixing $\Phi$ and $\phi$ so that \eqref{Phif} holds we moreover have
\begin{enumerate}\setcounter{enumi}{\value{c}}
\item the range of $\Phi$ is contained in $\sB=\C_0(X)$ if and only if $\phi$ is a proper map,
\item $\phi$ has dense image if and only if $\Phi$ is injective,
\item\label{std4} $\phi$ is injective if and only if $\Phi$ has strictly dense range.
\end{enumerate}
\end{theorem}

The strict topology on a multiplier algebra mentioned in Theorem \ref{std} is described e.g.~in \cite[Section 2]{gen} or \cite[Chapter 1]{lance}. The proof of the above theorem is a simple exercise in elementary topology and we leave it to the reader (see e.g.~\cite[Exercises to Chapter 2]{W-O}).

Let $\sA$ be a \cst-algebra. The dual space $\sA^*$ is naturally a module over $\sA$ and we will denote the natural left action of $a\in\sA$ on $\varphi\in\sA^*$ by $a\cdot\varphi$, so that $(a\cdot\varphi)(b)=\varphi(ba)$ for all $b\in\sA$. Note that if $\sC\subset\B(\sH)$ then $\sC$ acts in a natural way on the functionals in $\B(\sH)_*=\K(\sH)^*$ and we have
\begin{equation}\label{comMod}
\sC\cdot\K(\sH)^*=\sC\K(\sH)\cdot\K(\sH)^*=\K(\sH)\cdot\K(\sH)^*=\K(\sH)^*.
\end{equation}
(all sets above are automatically closed by the Cohen factorization theorem).

For \cst-algebras $\sA$ and $\sB$ their minimal tensor product will be denoted by $\sA\tens\sB$. For von Neumann algebras $\sM$ and $\sN$ the von Neumann algebra tensor product of $\sM$ and $\sN$ will be denoted by $\sM\vtens\sN$. The tensor flip $a\tens{b}\mapsto{b}\tens{a}$ will be denoted by $\sigma$ regardless of which \cst-algebras are being considered. We will also use the same symbol ``$\tens$'' to denote tensor product of Hilbert spaces.

In \cite{gen} S.L.~Woronowicz introduced a very important notion of a \cst-algebra generated by elements which do not necessarily belong to it. We will use a crucial part of his theory dealing with \cst-algebras ``generated by a quantum family of multipliers''. Let $\sA$ and $\sC$ be \cst-algebras and let $T\in\M(\sC\tens\sA)$. By analogy with the classical situation (when $\sC$ is commutative) the element $T$ is referred to as a \emph{quantum family of elements} of $\M(A)$ labeled by the spectrum of $\sC$ (cf.~\cite[Formula (2.5)]{gen}).

\begin{definition}[{\cite[Definition 4.1]{gen}}]\label{DefGen}
Let $\sA$ and $\sC$ be \cst-algebras. We say that $\sA$ is \emph{generated by $T\in\M(\sC\tens\sA)$} if for any Hilbert space $\sH$, any representation $\rho$ of $\sA$ on $\sH$ and any \cst-algebra $\sB\subset\B(\sH)$ the condition that $(\id\tens\rho)(T)\in\M(\sC\tens\sB)$ implies that $\rho\in\Mor(\sA,\sB)$.
\end{definition}

Examples of the situation described in Definition \ref{DefGen} are plentiful. For the simplest case consider a \cst-algebra $\sA$ generated by a finite set of elements $a_1,\dotsc,a_n\in\sA$ (in the usual sense, i.e.~the closure of the set of algebraic combinations of the elements $a_1,\dotsc,a_n$ and their adjoints coincides with $\sA$). Then $\sA$ is generated by $T\in\M(\CC^n\tens\sA)$ with
\begin{equation*}
T=\sum_{i=1}^ne_i\tens{a_i},
\end{equation*}
where $\{e_1,\dotsc,e_n\}$ is the standard basis of $\CC^n$. More complicated examples of \cst-algebras generated by quantum families of multipliers are given in \cite[Section 4]{gen}. In this paper we will be mostly interested in examples of this situation arising from representations of locally compact quantum groups to be studied in Subsection \ref{QGsub}, Section \ref{Reps} and Section \ref{Subgroups}.

\begin{remark}\label{uniquegen}
Let $\sH$ and $\sK$ be Hilbert spaces and consider \cst-algebras $\sA_1,\sA_2\subset\B(\sH)$ and $\sC\subset\B(\sK)$. Suppose that $T\in\B(\sK\tens\sH)$ is such that $T\in\M(\sC\tens\sA_1)\cap\M(\sC\tens\sA_2)$ and $T$ generates both $\sA_1$ and $\sA_2$. Then $\sA_1=\sA_2$, as the identity representation of $\sA_1$ is a morphism in $\Mor(\sA_1,\sA_2)$, and similarly the identity representation of $\sA_2$ is a morphism in $\Mor(\sA_2, \sA_1)$. This argument appeared already in \cite{gen}.
\end{remark}

Usually it is difficult to check that a given $T\in \M(\sC\tens\sA)$ generates $\sA$. For the needs of this paper it will be very useful to apply the following criterion. Note that if $T\in\M(\sC\tens\sA)$ and $\sC\subset\B(\sH)$, then each functional $\omega\in\B(\sH)_*$ defines an element of $\sC^*$, so that, in particular, $(\omega\tens\id)(T)\in\M(\sA)$ (\cite[Proposition 8.3]{lance}, \cite[Lemma A.3]{mnw}).

\begin{lemma}\label{cr}
Let $\sA$ and $\sC$ be \cst-algebras with $\sC\subset\B(\sH)$ for a Hilbert space $\sH$. Let $T\in\M(\sC\tens\sA)$ be unitary and define
\begin{equation*}
S=\bigl\{(\omega\tens\id)(T)\st\omega\in\B(\sH)_*\bigr\}\subset\M(\sA).
\end{equation*}
If $S\subset\sA$ and $S$ generates $\sA$ (as a subset of the \cst-algebra $\sA$) then $T\in\M(\sC\tens\sA)$ generates $\sA$.
\end{lemma}

\begin{proof}
Let $\sK$ be a Hilbert space and let $\rho$ be a representation of $\sA$ on $\sK$ such that $(\id\tens\rho)(T)\in\M(\sC\tens\sB)$ for a certain \cst-algebra $\sB\subset\B(\sK)$. It is easily seen that $\rho(S)\subset\M(\sB)$, which implies that $\rho(\sA)\subset\M(\sB)$ because $S$ generates $\sA$. Furthermore
\begin{equation*}
\begin{split}
\bigl(\rho(S)\sB\bigr)^{\nc}
&=\,
\bigr\{\rho\bigl((c\cdot\omega\tens\id)(T)\bigr)b\st{c\in\sC,\:b\in\sB,\:\omega\in\B(\sH)_*}\bigr\}^{\nc}\\
&=\bigl\{(\omega\tens\id)\bigl((\id\tens\rho)(T)\,(c\tens{b})\bigr)\st
{c\in\sC,\:b\in\sB,\:\omega\in\B(\sH)_*}\bigr\}^{\nc}\\
&=\bigl\{(\omega\tens\id)(c\tens{b})\st
{c\in\sC,\:b\in\sB,:\omega\in\B(\sH)_*}\bigr\}^{\nc}=\sB,
\end{split}
\end{equation*}
where in the first equality we used the formula \eqref{comMod} and in the last one we used the fact that $(\id\tens\rho)(T)$ is unitary in $\M(\sC\tens\sB)$. This shows that $\rho\in\Mor(\sA,\sB)$ and ends the proof.
\end{proof}

\begin{remark}\label{order}
Sometimes it is important to use the notion of a \cst-algebra generated by a quantum family of multipliers in a different version. More precisely let $\sA$ and $\sC$ be \cst-algebras and let $T\in\M(\sA\tens\sC)$ (note the different order of tensor factors from the one in Definition \ref{DefGen}). We will say that \emph{$T\in\M(\sA\tens\sC)$ generates $\sA$} if $\sigma(T)\in\M(\sC\tens\sA)$ generates $\sA$ in the sense described in Definition \ref{DefGen}. It can happen that a given $T\in\M(\sC\tens\sA)$ generates $\sA$ and at the same time $T\in\M(\sC\tens\sA)$ generates $\sC$. Coming back to the analogy with classical situation we would say that in the first statement $T$ is a quantum family of multipliers of $\sA$ labeled by the spectrum of $\sC$ and in the second statement $T$ is a quantum family of multipliers of $\sC$ labeled by the spectrum of $\sA$.
\end{remark}

Throughout the paper we will use the so-called \emph{leg-numbering notation}. This is explained in a number of texts on quantum groups, e.g.~\cite{pw,bs}.

\subsection{Locally compact quantum groups and their universal versions}\label{QGsub}

For the theory of locally compact quantum groups we refer the reader to \cite{KV} and to \cite{mnw} for an equivalent approach with different initial axioms. Most results of this paper are true in a potentially more general setting of quantum groups defined by modular multiplicative unitaries (\cite{modmu,WorSol,MRW}), but we will stay within the theory of Kustermans and Vaes. For a locally compact quantum group $\GG$ the corresponding \cst-algebra of ``continuous functions on $\GG$ vanishing at infinity'' will be denoted by $\C_0(\GG)$. This \cst-algebra is equipped with a comultiplication $\Delta_\GG\in\Mor\bigl(\C_0(\GG),\C_0(\GG)\tens\C_0(\GG)\bigr)$. There is also the \emph{reduced bicharacter} $\ww^{\GG}\in\M\bigl(\C_0(\hh{\GG})\tens\C_0(\GG)\bigr)$
(see [17, Page 53]), where $\hh{\GG}$ denotes the \emph{dual} of $\GG$. The Haar weights provide a realization of both $\C_0(\GG)$ and $\C_0(\hh{\GG})$ on the Hilbert space $L^2(\GG)$. Then $\ww^\GG\in\B\bigl(L^2(\GG)\tens{L^2(\GG)}\bigr)$ is a \emph{multiplicative unitary} (\cite{bs}) called the \emph{Kac-Takesaki operator} of $\GG$ (\cite{mnw}). The comultiplication is then implemented by $\ww^\GG$:
\begin{equation*}
\Delta_\GG(f)=\ww^\GG(f\tens\I)(\ww^\GG)^*
\end{equation*}
for all $f\in\C_0(\GG)$ (note that we are using the conventions of \cite{bs,mu,mnw,WorSol,MRW} favoring right Haar weights over left ones). The embedding of $\C_0(\GG)$ into $\B\bigl(L^2(\GG)\bigr)$ defines also the von Neumann algebra $L^\infty(\GG)$ as $\C_0(\GG)''$. Moreover we have
\begin{equation}\label{Slices}
\C_0(\GG)=\bigl\{(\omega\tens\id)(\ww^\GG)\St\omega\in\B\bigl(L^2(\GG)\bigr)_*\bigr\}^{\nc}.
\end{equation}
In fact $\C_0(\GG)$ is generated by the quantum family $\ww^\GG\in\M\bigl(\C_0(\hh{\GG})\tens\C_0(\GG)\bigr)$ in the sense described in Definition \ref{DefGen} (\cite{mu}). Moreover the \cst-algebra $\C_0(\hh{\GG})$ is generated by quantum family $\ww^\GG\in\M\bigl(\C_0(\hh{\GG})\tens\C_0(\GG)\bigr)$ (note the difference, cf.~Remark \ref{order}).

The dense subspace
\begin{equation}\label{FAlg}
\cA_\GG=\bigl\{(\omega\tens\id)(\ww^\GG)\St\omega\in\B\bigl(L^2(\GG)\bigr)_*\bigr\}\subset\C_0(\GG)
\end{equation}
(no closure) is called the \emph{Fourier algebra} of $\GG$ (\cite[Section 8]{disser}). Note that the vector space $\cA_\GG$ is indeed a subalgebra of $\C_0(\GG)$ (\cite[Proposition 1.4]{bs}). We will identify the quotient of $\B\bigl(L^2(\GG)\bigr)_*$ by the functionals which vanish on $\C_0(\GG)$ with $L^\infty(\GG)_*$. It is clear that one can use this space of functionals instead of $\B\bigl(L^2(\GG)\bigr)_*$ in all formulas of the form \eqref{FAlg} or \eqref{Slices}.

\begin{lemma}\label{easy}
Let $\GG$ be a quantum group and let $\eta\in\C_0(\GG)^*$ be non-zero. Then $(\id\tens\eta)(\ww^\GG)\in\M\bigl(\C_0(\hh{\GG})\bigr)$ is non-zero.
\end{lemma}

\begin{proof}
If $\eta\neq{0}$ then it must be non-zero on the norm dense set $\cA_\GG$. Therefore there is a normal functional $\omega$ on $\B\bigl(L^2(\GG)\bigr)$ such that $\eta\bigl((\omega\tens\id)(\ww^\GG)\bigr)\neq{0}$. Consequently
\begin{equation*}
\omega\bigl((\id\tens\eta)(\ww^\GG)\bigr)\neq{0}
\end{equation*}
which clearly implies that $(\id\tens\eta)(\ww^\GG)\neq{0}$.
\end{proof}

The first consequence of Lemma \ref{easy} is that $\cA_\GG$ is isomorphic as a vector space to $L^\infty(\hh{\GG})_*$; in particular it is a Banach space with the norm transported from $L^\infty(\hh{\GG})_*$. Indeed this is the point of view of classical harmonic analysis (\cite{eymard}). We will view the Fourier algebra both as a Banach space and a subspace of $\C_0(\GG)$.

The universal object related to $\GG$ is a \cst-algebra which we will denote by $\C_0^\uu(\GG)$, endowed with a comultiplication $\Delta_\GG^\uu\in\Mor\bigl(\C_0^\uu(\GG),\C_0^\uu(\GG)\tens\C_0^\uu(\GG)\bigr)$. This object was introduced and analyzed in \cite{Johanuniv}. In the more general setting of quantum groups defined by modular multiplicative unitaries the universal \cst-algebra corresponding to $\GG$ is studied in \cite[Section 5]{WorSol}. The reduced bicharacter \emph{lifts} to the universal level, i.e.~we have the \emph{universal bicharacter}
\begin{equation*}
\WW^\GG\in\M\bigl(\C_0^\uu(\hh{\GG})\tens\C_0^\uu(\GG)\bigr)
\end{equation*}
(\cite[Proposition 3.8]{Johanuniv} and \cite[Proposition 4.8]{MRW}). Following the conventions of \cite{WorSol} the \emph{reducing morphisms} for $\GG$ and $\hh{\GG}$ will be denoted by $\Lambda_\GG\in\Mor\bigl(\C_0^\uu(\GG),\C_0(\GG)\bigr)$ and $\Lambda_{\hh{\GG}}\in\Mor\bigl(\C_0^\uu(\hh{\GG}),\C_0(\hh{\GG})\bigr)$ respectively (see \cite[Definition 35]{WorSol}). We have
\begin{equation}\label{down}
(\Lambda_{\hh{\GG}}\tens\Lambda_\GG)(\WW^\GG)=\ww^\GG.
\end{equation}
The elements $(\id\tens\Lambda_\GG)(\WW^\GG)$ and $(\Lambda_{\hh{\GG}}\tens\id)(\WW^\GG)$ will be denoted by
\begin{equation}\label{wWWw}
\Ww^\GG\in\M\bigl(\C_0^\uu(\hh{\GG})\tens\C_0(\GG)\bigr)\qquad\text{and}\qquad
\wW^\GG\in\M\bigl(\C_0(\hh{\GG})\tens\C_0^\uu(\GG)\bigr)
\end{equation}
respectively. We have
\begin{equation}\label{C0u}
\C_0^\uu(\GG)=\bigl\{(\omega\tens\id)(\wW^\GG)\St\omega\in\B\bigl(L^2(\GG)\bigr)_*\bigr\}^{\nc}
\end{equation}
(\cite[Formula (5.14)]{WorSol}) and consequently the \cst-algebra $\C_0^\uu(\GG)$ is generated by the quantum family $\wW^\GG\in\M\bigl(\C_0(\hh{\GG})\tens\C_0^\uu(\GG)\bigr)$ (by Lemma \ref{cr}, cf.~\cite{WorSol} and Proposition \ref{generatingAU}).

The universal dual is determined by the quantum group $\GG$ only up to isomorphism, so when $\Lambda_{\GG}$ is an isomorphism (i.e.~$\GG$ is \emph{coamenable}) then we can declare that $\C_0^\uu(\GG)=\C_0(\GG)$ and $\Lambda_\GG=\id$. Then
\begin{equation*}
\WW^\GG=\wW^\GG\quad\text{and}\quad\wW^{\hh{\GG}}=\ww^{\hh{\GG}}.
\end{equation*}
Similarly, when $\hh{\GG}$ is coamenable then
\begin{equation}\label{forHatH}
\WW^{\hh{\GG}}=\wW^{\hh{\GG}}\quad\text{and}\quad\wW^{\GG}=\ww^{\GG}.
\end{equation}
Note that quantum groups which are \emph{classical} (i.e.~quantum groups $\GG$ for which $\C_0(\GG)$ is commutative) are always coamenable.

\begin{proposition}\label{LamSli}
Let $\GG$ be a locally compact quantum group. Then the reducing map $\Lambda_{\GG}$ is injective on the subspace
\begin{equation}\label{BG}
\bigl\{(\id\tens\eta)(\Ww^{\hh{\GG}})\St\eta\in\B\bigl(L^2(\hh{\GG})\bigr)_*\bigr\}.
\end{equation}
\end{proposition}

The proof is obvious:
\begin{equation*}
\Bigl((\id\tens\eta)(\Ww^{\hh{\GG}})\neq{0}\Bigr)\quad\Longrightarrow\quad\bigl(\eta\neq{0}\bigr)
\quad\Longrightarrow\quad\Bigl((\id\tens\eta)(\ww^{\hh{\GG}})\neq{0}\Bigr)
\end{equation*}
by Lemma \ref{easy} applied to $\hh{\GG}$. Note that Proposition \ref{LamSli} can be viewed as a generalization of the very useful \cite[Proposition 3.2]{cqg}. The image of $\Lambda_\GG$ on the subspace \eqref{BG} is exactly the Fourier algebra $\cA_\GG$. It follows that $\Lambda_{\GG}\bigl(\C_0^\uu(\GG)\bigr)=\C_0(\GG)$.

The \emph{Fourier-Stieltjes algebra} of $\GG$ is the space
\begin{equation*}
\cB_\GG=
\bigl\{(\eta\tens\id)(\WW^\GG)\St\eta\in\C_0^\uu(\hh{\GG})^*\bigr\}\subset\M\bigl(\C_0^\uu(\GG)\bigr)
\end{equation*}
(see \cite[Section 8]{disser}, note that in that paper $\cB_\GG$ was embedded into $\M\bigl(\C_0(\GG)\bigr)$ and not into $\M\bigl(\C_0^\uu(\GG)\bigr)$). A reasoning analogous to that in the proof of Lemma \ref{easy} shows that $\cB_\GG$ is isomorphic as a vector space to $\C_0^\uu(\hh{\GG})^*$. Indeed, as $\bigl\{\bigl(\id\tens[\omega\comp\Lambda_\GG]\bigr)(\Ww^{\GG})\St\omega\in{L^\infty(\GG)_*}\bigr\}$
is dense in $\C_0^\uu(\hh{\GG})$ (\cite[Section 5]{WorSol}), a non-zero $\eta$ must be non-zero on some element of the form $\bigl(\id\tens[\omega\comp\Lambda_\GG]\bigr)(\Ww^{\GG})$, so $(\omega\comp\Lambda_\GG)\bigl((\eta\tens\id)(\WW^{\GG})\bigr)\neq{0}$. In particular $(\eta\tens\id)(\WW^{\GG})\neq{0}$.

In what follows we shall utilize both pictures of $\cA_\GG$ and $\cB_\GG$ --- as Banach spaces of functionals and at the same time as (non-closed) subspaces of $\C_0(\GG)$ and $\M\bigl(\C_0^\uu(\GG)\bigr)$ respectively.

A quantum group $\GG$ is \emph{compact} if the \cst-algebra $\C_0(\GG)$ is unital. In this case we write $\C(\GG)$ instead of $\C_0(\GG)$. Dually, $\GG$ is \emph{discrete} if $\hh{\GG}$ is compact. In this case $\C_0(\GG)$ is a $\cc_0$-direct sum of matrix algebras and we write $\cc_0(\GG)$ instead of $\C_0(\GG)$. We also write in this case $\ell^\infty(\GG)$ for $L^\infty(\GG)$. Discrete quantum groups are always coamenable (\cite{pw}). We refer to \cite{cqg} for the complete account of the theory of compact quantum groups and to \cite[Section 3]{pw} for a thorough treatment of discrete quantum groups.

Finally let us mention that on the level of bicharacters the duality between $\GG$ and $\hh{\GG}$ is implemented by the tensor flip and the adjoint operation:
\begin{equation}\label{DualBigW}
\WW^{\hh{\GG}}=\sigma({\WW^\GG})^*.
\end{equation}
It follows that
\begin{equation*}
\wW^{\hh{\GG}}=\sigma({\Ww^\GG})^*\qquad\text{and}\qquad
\ww^{\hh{\GG}}=\sigma({\ww^\GG})^*.
\end{equation*}

\subsection{Homomorphisms of locally compact quantum groups}\label{morphism}

Let $\GG$ and $\HH$ be locally compact quantum groups. In \cite{MRW} it is shown that the following three classes of objects are in a one-to-one correspondence:
\begin{enumerate}
\item\emph{strong quantum homomorphisms}: morphisms
\begin{equation*}
\pi\in\Mor\bigl(\C_0^\uu(\GG),\C_0^\uu(\HH)\bigr)
\end{equation*}
such that
\begin{equation*}
(\pi\tens\pi)\comp\Delta_\GG^\uu=\Delta_\HH^\uu\comp\pi;
\end{equation*}
\item\emph{bicharacters (from $\HH$ to $\GG$)}: unitaries
\begin{equation*}
V\in\M\bigl(\C_0(\hh{\GG})\tens\C_0(\HH)\bigr)
\end{equation*}
such that
\begin{equation}\label{DelDel}
\begin{split}
(\Delta_{\hh{\GG}}\tens\id_{\C_0(\HH)})(V)&=V_{23}V_{13},\\
(\id_{\C_0(\hh{\GG})}\tens\Delta_{\HH})(V)&=V_{12}V_{13}.
\end{split}
\end{equation}
\item\emph{right quantum homomorphisms}: morphisms
\begin{equation*}
\rho\in\Mor\bigl(\C_0(\GG),\C_0(\GG)\tens\C_0(\HH)\bigr)
\end{equation*}
such that
\begin{equation*}
\begin{split}
(\Delta_{\GG}\tens\id)\comp\rho&=(\id\tens\rho)\comp\Delta_{\GG},\\
(\id\tens\Delta_{\HH})\comp\rho&=(\rho\tens\id)\comp\rho.
\end{split}
\end{equation*}
\end{enumerate}

All these should be thought of as alternative descriptions of a fixed homomorphism from $\HH$ to $\GG$. Note that the reduced bicharacter $\ww^\GG$ of $\GG$ introduced in Subsection \ref{QGsub} is a bicharacter from $\GG$ to $\GG$ in the above sense and describes the identity homomorphism. Sometimes, to simplify the language, we will refer to a strong quantum homomorphism $\pi$ as above as a homomorphism from $\HH$ to $\GG$. A strong quantum homomorphism $\pi$ is related to the bicharacter $V$ via the formula
\begin{equation}\label{Vdef}
V=\bigl(\Lambda_{\hh{\GG}}\tens[\Lambda_{\HH}\comp\pi]\bigr)(\WW^{\GG}),
\end{equation}
while the right quantum homomorphism $\rho$ is given by
\begin{equation*}
\rho(x)=V(x\tens\I_{\C_0(\HH)})V^*
\end{equation*}
for any $x\in\C_0(\GG)\subset\B\bigl(L^2(\GG)\bigr)$.

One can also check (see \cite[Lemma 3.4]{MRW}) that for a unitary  $V\in\M\bigl(\C_0(\hh{\GG})\tens\C_0(\HH)\bigr)$ the conditions \eqref{DelDel} are equivalent to the following ``twisted'' pentagonal equations:
\begin{subequations}
\begin{align}
V_{23}\ww^{\GG}_{12}&=\ww^{\GG}_{12}V_{13}V_{23},&\text{in}\quad&
\M\bigl(\C_0(\hh{\GG})\tens\K\bigl(L^2(\GG)\bigr)\tens\C_0(\HH)\bigr),\label{VG}\\
\ww^{\HH}_{23}V_{12}&=V_{12}V_{13}\ww^{\HH}_{23},&\text{in}\quad&
\M\bigl(\C_0(\hh{\GG})\tens\K\bigl(L^2(\HH)\bigr)\tens\C_0(\HH)\bigr).\label{VH}
\end{align}
\end{subequations}

The next result, namely \cite[Proposition 3.14]{MRW}, describes in the simplest way the construction of the \emph{dual homomorphisms} (cf.~\eqref{DelDel}).

\begin{proposition}\label{dualprop}
If $V$ is a bicharacter from $\HH$ to $\GG$, the unitary $\hh{V}=\sigma(V^*)\in\M\bigl(\C_0(\HH)\tens\C_0(\hh{\GG})\bigr)$ is a bicharacter from $\hh{\GG}$ to $\hh{\HH}$.
\end{proposition}

Proposition \ref{dualprop} makes possible the following definition:

\begin{definition}
Let $\pi$ be a morphism from $\HH$ to $\GG$ with corresponding bicharacter $V$. Then the strong quantum homomorphism defined by $\hh{V}$ is called the \emph{dual} of $\pi$ and will be denoted by $\hh{\pi}$, so that
$\hh{\pi}\in\Mor\bigl(\C_0^\uu(\hh{\HH}),\C_0^\uu(\hh{\GG})\bigr)$.
\end{definition}

Let us note the most fundamental equality relating $\pi$ to $\hh{\pi}$ (and determining $\hh{\pi}$ uniquely) contained in \cite[Theorem 4.15]{MRW}:
\begin{equation}\label{basic}
(\id\tens\pi)(\WW^{\GG})=(\hh{\pi}\tens\id)(\WW^{\HH}).
\end{equation}
By applying $(\Lambda_{\hh{\GG}}\tens\Lambda_{\HH})$ to both sides and using \eqref{Vdef} we obtain
\begin{equation}\label{this}
\bigl(\Lambda_{\hh{\GG}}\tens[\Lambda_{\HH}\comp\pi]\bigr)(\WW^{\GG})=V
=\bigl([\Lambda_{\hh{\GG}}\comp\hh{\pi}]\tens\Lambda_{\HH}\bigr)(\WW^{\HH}).
\end{equation}
Moreover if $\pi_1$ and $\pi_2$ are strong quantum homomorphisms associated with homomorphisms from $\GG_1$ to $\GG_2$ and from $\GG_2$ to $\GG_3$ respectively then
\begin{equation}\label{pipi}
\hh{\pi_2}\comp\hh{\pi_1}=\hh{\pi_1\comp\pi_2},
\end{equation}
since \eqref{basic} characterizes the dual strong quantum homomorphism. Thus if $\pi$ is an isomorphism of \cst-algebras then so is $\hh{\pi}$.

\begin{theorem}\label{isoT}
Let $\GG$ and $\HH$ be locally compact quantum groups. Consider a homomorphism from $\HH$ to $\GG$ such that the corresponding $\pi\in\Mor\bigl(\C_0^\uu(\GG),\C_0^\uu(\HH)\bigr)$ is an isomorphism, i.e.~$\pi$ is a one-to-one map from $\C_0^\uu(\GG)$ onto $\C_0^\uu(\HH)$. Then there exists an isomorphism $\pir$ of $\C_0(\GG)$ onto $\C_0(\HH)$ such that $\pir\comp\Lambda_{\GG}=\Lambda_\HH\comp\pi$.
\end{theorem}

Theorem \ref{isoT} says that isomorphisms in the category of locally compact quantum groups considered in \cite{MRW} drop down to \cst-algebraic isomorphisms of the reduced level. In what follows we will refer to this situation by simply saying that $\GG$ and $\HH$ are isomorphic. A proof of this result may be given along the lines of \cite[Proposition 8.7]{Johanuniv} (cf.~also \cite[Proposition 7.1]{Johanuniv}). In Section \ref{Subgroups} we will give a short proof of Theorem \ref{isoT} using representation theory of locally compact quantum groups and techniques developed in this paper. Let us note that these techniques make no use of the existence of Haar weights and are equally applicable to quantum groups arising from modular multiplicative unitaries.

\section{Representations of locally compact quantum groups}\label{Reps}

In this section we recall some basic notions of the representation theory of locally compact quantum groups (\cite{bs}, \cite[Section 3]{WorSol}) and establish alternative characterizations of quasi-equivalence of two representations of a given quantum group (Theorem \ref{QE = Fourier}).

Let $\GG$ be a locally compact quantum group and let $\sH$ be a Hilbert space. A \emph{strongly continuous unitary representation} of $\GG$ on $\sH$ is a unitary element $U\in\M\bigl(\K(\sH)\tens\C_0(\GG)\bigr)$ such that
\begin{equation*}
(\id\tens\Delta_\GG)(U)=U_{12}U_{13}.
\end{equation*}
We will usually write simply of ``representations of $\GG$''. Moreover the Hilbert space $\sH$ will be usually decorated by the subscript $U$, so that $U\in\M\bigl(\K(\sH_U)\tens\C_0(\GG)\bigr)$.

For such a representation $U$ of $\GG$ the subspace
\begin{equation*}
\sA_U=\bigl\{(\id\tens\omega)(U)\st\omega\in{L^\infty(\GG)_*}\bigr\}^\nc
\end{equation*}
is a non-degenerate \cst-subalgebra of $\B(\sH)$ (it was denoted ``$B_U$'' in \cite{WorSol}). In fact $U$ is a multiplier of $\sA_U\tens\C_0(\GG)$ and the quantum family $U\in\M\bigl(\sA_U\tens\C_0(\GG)\bigr)$ generates $\sA_U$ (\cite{mu,WorSol}).

We will also use at some point the notation
\begin{equation*}
\cA_U=\bigl\{(\id\tens\omega)(U)\st\omega\in{L^\infty(\GG)_*}\bigr\}.
\end{equation*}
It is easy to check that $\cA_U$ is an algebra --- this is a quantum group analogue of the ``Fourier space of a representation'' defined in \cite[Definition (2.1)]{arsac}; note for example that $\cA_{\ww^{\GG}}$ is the Fourier algebra $\cA_{\hh{\GG}}$ of $\hh{\GG}$. Observe further that a bicharacter from $\HH$ to $\GG$ is a representation of $\HH$ on $L^2(\GG)$ and it follows from Proposition \ref{dualprop} that $\hh{V}$ is a representation of $\hh{\GG}$ on $L^2(\HH)$.

The generating property for representations can be reformulated in terms of their slices. In the following proposition note the use of the notion of a \cst-algebra generated by a quantum family of multipliers in the version described in Remark \ref{order}.

\begin{proposition} \label{generatingAU}
Let $U$ be a representation of $\GG$ on a Hilbert space $\sH$ and let $\sA$ be a non-degenerate \cst-subalgebra of $\B(\sH)$. Assume that $U\in\M\bigl(\sA\tens\C_0(\GG)\bigr)$. Then the following are equivalent:
\begin{enumerate}
\item $U\in\M\bigl(\sA\tens\C_0(\GG)\bigr)$ generates $\sA$;
\item $\sA=\sA_U$.
\end{enumerate}
\end{proposition}

\begin{proof}
A direct consequence of the fact that $U$ generates $\sA_U$ and Remark \ref{uniquegen} (cf.~\cite[Subsection 3.5]{WorSol}).
\end{proof}

The standard notions of representation theory were all collected in \cite[Section 3]{WorSol}.
\begin{itemize}
\item Two representations $U$ and $V$ of $\GG$ are \emph{equivalent} if there exists a unitary operator $T\in\B(\sH_U,\sH_V)$ such that
\begin{equation*}
(T\tens\I)U=V(T\tens\I).
\end{equation*}
\item If $\sH$ is a Hilbert space then the \emph{trivial representation} of $\GG$ on $\sH$ is
\begin{equation*}
\II_{\sH}=\I_{\B(\sH)}\tens\I_{\C_0(\GG)}\in\M\bigl(\K(\sH)\tens\C_0(\GG)\bigr).
\end{equation*}
\item The \emph{tensor product} of two representations $U$ and $V$ is the representation
\begin{equation*}
U\tp{V}\in\M\bigl(\K(\sH_U\tens\sH_V)\tens\C_0(\GG)\bigr)
\end{equation*}
defined by
\begin{equation*}
U\tp{V}=U_{13}V_{23}.
\end{equation*}
\item Representation $U$ and $V$ are \emph{quasi-equivalent} if there exists a Hilbert space $\sH$ such that $\II_\sH\tp{U}$ and $\II_\sH\tp{V}$ are equivalent (\cite[Proposition 13]{WorSol}).
\end{itemize}

The following theorem will be crucial in the next section, when we analyze a definition of a closed quantum subgroup proposed in \cite{Vaes}. The implication \eqref{QEF1}$\Rightarrow$\eqref{QEF2}
is \cite[Corollary 15]{WorSol}.

\begin{theorem}\label{QE = Fourier}
Let $U$ and $V$ be representations of $\GG$ on $\sH_U$ and $\sH_V$ respectively. The following three conditions are equivalent:
\begin{enumerate}
\item\label{QEF1} $U$ is quasi-equivalent to $V$;
\item\label{QEF2} there exists a (necessarily unique) normal $*$-isomorphism $\gamma\colon\sA_U''\to\sA_V''$ such that
\begin{equation*}
(\gamma\tens\id)(U)=V;
\end{equation*}
\item\label{QEF3} we have
\begin{equation}\label{iniii}
\bigl\{(\eta\tens\id_{\B(L^2(\GG))})(U)\st\eta\in\B(\sH_U)_*\bigr\}
=\bigl\{(\mu\tens\id_{\B(L^2(\GG))})(V)\st\mu\in\B(\sH_V)_*\bigr\}.
\end{equation}
\end{enumerate}
\end{theorem}

\begin{proof}
\eqref{QEF1}$\Rightarrow$\eqref{QEF2}. Suppose that $U$ and $V$ are quasi-equivalent. Let $\sK$ be a Hilbert space and let $T\colon\sK\tens\sH_U\to\sK\tens\sH_V$ be a unitary such that
\begin{equation}\label{qereg}
T_{12}U_{23}T_{12}^*=V_{23}.
\end{equation}
Take $\omega\in\B\bigl(L^2(\GG)\bigr)_*$ and put $x=(\id\tens\omega)(U)$ and $y=(\id\tens\omega)(V)$. Equation \eqref{qereg} shows that  $T(\I_\sK\tens{x})T^*=\I_\sK\tens{y}$. This implies that
$T(\I_{\sK}\tens\sA_U'')T^*\subset{\I_{\sK}\vtens\sA_V''}$. Applying a similar argument in the converse direction we observe that actually $T(\I_{\sK}\tens\sA_U'')T^*=\I_{\sK}\vtens\sA_V''$ so there exists
a normal $*$-isomorphism $\gamma\colon\sA_U''\to\sA_V''$ such that $\I_\sK\tens\gamma(x)=T(\I_\sK\tens{x})T^*$ for all $x\in\sA_U''$. Using equation \eqref{qereg} we see that $(\gamma\tens\id)(U)=V$.

\eqref{QEF2}$\Rightarrow$\eqref{QEF1}. Let $\gamma:\sA_U''\to\sA_V''$ be an normal $*$-isomorphism such that $(\gamma\tens\id)(U)=V$. It is a well known fact (see e.g.~\cite[Theorem III.2.2.8]{bl})
that $\gamma$ is of the form $\I_{\sK}\tens\gamma(x)=T(\I_{\sK}\tens{x})T^*$ for some Hilbert space $\sK$ and a unitary operator $T\colon\sK\tens\sH_U\to\sK\tens\sH_V$. It is then easy to check that $T_{12}U_{23}T_{12}^*=V_{23}$, which proves the quasi-equivalence of $U$ and $V$.

\eqref{QEF2}$\Rightarrow$\eqref{QEF3}. Since for any $\mu\in\B(\sH_V)_*$ the composition $\mu\comp\gamma$ is a normal functional on $\sA_V''$, there exists $\eta\in\B(\sH_V)_*$ such that $\mu\comp\gamma=\bigl.\eta\bigr|_{\sA_V''}$. This shows that
\begin{equation*}
\bigl\{(\eta\tens\id_{\B(L^2(\GG))})(U)\st\omega\in\B(\sH_U)_*\bigr\}\supset
\bigl\{(\mu\tens\id_{\B(L^2(\GG))})(V)\st\mu\in\B(\sH_V)_*\bigr\}.
\end{equation*}
Exchanging the roles of $U$ and $V$ we get the opposite inclusion; hence \eqref{QEF3} follows.

\eqref{QEF3}$\Rightarrow$\eqref{QEF2}.
Let $\tilde{\kappa}$ be the extension of the antipode $\kappa$ of $\GG$ to an unbounded operator acting on $\M\bigl(\C_0(\GG)\bigr)$ (see \cite[Theorem 1.6]{mu}). Recall that $\tilde{\kappa}$ is a densely defined operator acting on its domain $\D(\tilde{\kappa})\subset\M\bigl(\C_0(\GG)\bigr)$ such that for any representation $U\in\M\bigl(\K(\sH_U)\tens\C_0(\GG)\bigr)$ of $\GG$ and any $\eta\in\B(\sH_U)_*$ we have $(\eta\tens\id)(U)\in\D(\tilde{\kappa})$ and
\begin{equation}\label{kap}
\tilde{\kappa}\bigl((\eta\tens\id)(U)\bigr)=(\eta\tens\id)(U^*).
\end{equation}
Consider the set $X\subset{L^\infty(\GG)_*}$ defined so that $\omega\in{X}$ if and only if $\omega^*\comp\tilde{\kappa}$ extends to a bounded normal functional on $L^\infty(\GG)$. Define further $\cA_U^X=\bigl\{(\id\tens\omega)(U)\st\omega\in{X}\bigr\}$. Equation \eqref{kap} and the fact that $U$ is a representation ensures that $\cA_U^X$ is a weakly dense $*$-subalgebra of $\sA_U''$.

Let us define a map $\gamma_0\colon\cA_U^X\to\sA_V''$ by the following formula:
\begin{equation*}
\gamma_0\bigl((\id\tens\omega)(U)\bigr)=(\id\tens\omega)(V),\qquad\omega\in{X}.
\end{equation*}
Fix $\omega\in{X}$. Since $V\in\M\bigl(\sA_V\tens\C_0(\GG)\bigr)$, the expression $(\id\tens\omega)(V)$ makes sense. Moreover if $(\id\tens\omega)(U)=0$ then for any $\eta\in\B(\sH_U)_*$ we have $\omega\bigl((\eta\tens\id)(U)\bigr)=0$ and by our assumption $\omega\bigl((\mu\tens\id)(V)\bigr)=0$ for any $\mu\in\B(\sH_V)_*$. The last property means that $(\id\tens\omega)(V)=0$ and shows that $\gamma_0$ is well-defined (cf.~Lemma \ref{easy}). It can be checked that $\gamma_0$ is a $*$-homomorphism, for example
\begin{equation*}
\gamma_0\bigl((\id\tens\omega)(U)\bigr)^*=\bigl((\id\tens\omega)(V)\bigr)^*=
\bigl(\id\tens[\omega^*\comp\tilde{\kappa}]\bigr)(V)
=\gamma_0\bigl(\bigl((\id\tens\omega)(U)\bigr)^*\bigr).
\end{equation*}

In the next step we shall show that $\gamma_0$ may be extended to a normal $*$-isomorphism $\gamma\colon\sA_U''\to\sA_V''$. Take $x\in\sA_U$. Using Kaplansky's density theorem, we may find a bounded net $(x_i)_{i\in\cI}$ of elements in $\cA_U^X$, say $x_i=(\id\tens\omega_i)(U)$ with $\omega_i\in{X}$ such that
$\mathrm{w}\text{-}\!\lim\limits_{i\in\cI}x_i=x$. Let $M\in\RR_+$ be the corresponding bound: $\|x_i\|\leq{M}$. In what follows we shall prove that $\bigl(\gamma_0(x_i)\bigr)_{i\in\cI}$ weakly converges to a certain element $y\in\sA_V''$. Take now $\mu\in(\sA_{V}'')_*$. For each $i\in\cI$ we have $\mu\bigl(\gamma_0(x_i)\bigr)=\omega_i\bigl((\mu\tens\id)(V)\bigr)$. For $\eta\in(\sA''_{U})_*$ such that $(\mu\tens\id)(V)=(\eta\tens\id)(U)$ we obtain $\mu\bigl(\gamma_0(x_i)\bigr)=\eta(x_i)$. In particular $\bigl|\mu\bigl(\gamma_0(x_i)\bigr)\bigr|\leq{M}\|\eta\|$ and $\lim\limits_{i\in\cI}\mu\bigl(\gamma_0(x_i)\bigr)=\eta(x)$. Interpreting $\sA_V''$ as the dual of
$(\sA_{V}'')_*$ we conclude that the family $\bigl(\gamma_0(x_i)\bigr)_{i\in\cI}$ of functionals on $(\sA_{V}'')_*$  is pointwise bounded. By the Banach-Steinhaus theorem it is norm bounded. Let $N\in\RR_+$ be a bound: $\bigl\|\gamma_0(x_i)\bigr\|\leq{N}$ for all $i\in\cI$. Noting that the map $(\sA_{V}'')_*\ni\mu\mapsto\lim\limits_{i\in\cI}\,\mu\bigl(\gamma_0(x_i)\bigr)\in\CC$ is a bounded functional with the norm not greater than $N$ we conclude the existence of $y\in\sA_V''$, such that
$y=\mathrm{w}\text{-}\!\lim\limits_{i\in\cI}\gamma_0(x_i)$. This enables us to define the aforementioned extension by putting $\gamma(x)=y$. If $x_i\xrightarrow[i\in\cI]{}0$ then for each $\mu$ as above $\lim\limits_{i\in\cI}\mu\bigl(\gamma_0(x_i)\bigr)=\lim\limits_{i\in\cI}\eta(x_i)=0$, so that $y=0$. This implies that $\gamma$ is well defined.

The equality $\gamma(x^*)=\gamma(x)^*$ for any $x\in\cA_U^X$ and the fact that the star operation is weakly continuous imply that $\gamma(x^*)=\gamma(x)^*$ for any $x\in\sA_U''$. We will now show using once again \eqref{iniii} that for any $x,x'\in\sA_U''$ we have $\gamma(xx')=\gamma(x)\gamma(x')$. Note that, in the notation of the previous paragraph, for any $i\in\cI$ we have $\mu\bigl(\gamma_0(x_i)\bigr)=\eta(x_i)$ for a certain $\eta\in(\sA''_{U})_*$. Passing to the limit we get $\mu\bigl(\gamma(x)\bigr)=\eta(x)$, for any $x\in\sA_U''$. Note also that
\begin{equation*}
(\eta\cdot{x})(x'_i)=\eta(xx'_i)=\mu\bigl(\gamma(xx'_i)\bigr)
=\bigl(\mu\cdot\gamma(x)\bigr)\bigl(\gamma(x'_i)\bigr)
\end{equation*}
for any $x,x'_i\in\cA_U^X$. Again, passing to the limit, we get
\begin{equation*}
(\eta\cdot{x})(x')=\bigl(\mu\cdot\gamma(x)\bigr)\bigl(\gamma(x')\bigr)
\end{equation*}
for any $x\in\cA_U^X$ and $x'\in\sA_U''$. Replacing $x\in\cA_U^X$ with a bounded, weakly convergent net $(x_i)$ of  elements of $\cA_U^X$ and passing to the limit yields $(\eta\cdot{x})(x')=\bigl(\mu\cdot\gamma(x)\bigr)\bigl(\gamma(x')\bigr)$ for any $x,x'\in\sA_U''$. Finally we compute:
\begin{equation*}
\mu\bigl(\gamma(xx')\bigr)=\eta(xx')=(\eta\cdot{x})(x')=
\bigl(\mu\cdot\gamma(x)\bigr)\bigl(\gamma(x')\bigr)=\mu\bigl(\gamma(x)\gamma(x')\bigr)
\end{equation*}
which shows that $\gamma(xx')=\gamma(x)\gamma(x')$ for any $x,x'\in\sA_U''$.

Exchanging the roles of $U$ and $V$ leads to the inverse $*$-homomorphism $\gamma^{-1}\colon\sA_V''\to\sA_U''$. This shows that $\gamma$ is normal, since isomorphisms of von Neumann algebras are automatically normal (\cite[Corollary 3.10, page 135]{Tak}).
\end{proof}

It was shown in \cite{WorSol} that a unitary representation $U$ of $\GG$ is quasi-equivalent to $\ww^{\GG}$ if it is \emph{right absorbing}, i.e.~for any other representation $V$ of $\GG$ the tensor product $V\tp{U}$ is equivalent to $\II_{\sH_V}\tp{U}$ (this can be viewed as a version of the Fell absorption principle). We finish the section with a proposition which describes relation between quasi-equivalence of a given representation $U$ of $\GG$ with $\ww^{\GG}$ and the fact that $U\in\M\bigl(\K(\sH_U)\tens\C_0(\GG)\bigr)$ generates $\C_0(\GG)$.

\begin{proposition}\label{QE -> generating}
Let $U$ be a  representation of $\GG$ quasi-equivalent to $\ww^{\GG}$. Then the unitary element $U\in\M\bigl(\sA_U\tens\C_0(\GG)\bigr)$ generates $\C_0(\GG)$. On the other hand a representation $U$ which generates $\C_0(\GG)$ need not be quasi-equivalent to $\ww^{\GG}$ (even when $\GG$ is a locally compact group).
\end{proposition}

\begin{proof}
From Theorem \ref{QE = Fourier}\eqref{QEF3} it follows that
\begin{equation*}
\bigl\{(\omega\tens\id_{\B(L^2(\GG))})(U)\St\omega\in\B(\sH_U)_*\bigr\}
=\bigl\{(\mu\tens\id_{\B(L^2(\GG))})(\ww^{\GG})\St\mu\in\B\bigl(L^2(\GG)\bigr)_*\bigr\}.
\end{equation*}
Since for $Y=\bigl\{(\mu\tens\id_{\B(L^2(\GG))})(\ww^{\GG})\st\mu\in\B\bigl(L^2(\GG)\bigr)_*\bigr\}$ we have  $\bb{Y}=\C_0(\GG)$, we see that $Y$ generates $\C_0(\GG)$ as a \cst-algebra. Lemma \ref{cr} implies that $U\in\M\bigl(\sA_U\tens\C_0(\GG)\bigr)$ generates $\C_0(\GG)$.

For the second part it suffices to observe the following fact: let $U_1$ and $U_2$ be representations of a locally compact quantum group $\GG$ and let $U$ be their direct sum (\cite[Subsection 3.3.1]{WorSol}). If $U_1\in\M\bigl(\K(\sH_{U_1})\tens\C_0(\GG)\bigr)$ generates $\C_0(\GG)$ then so does $U\in\M\bigl(\K(\sH_{U_1}\oplus\sH_{U_2})\tens\C_0(\GG)\bigr)$.

Let $\GG=\ZZ$ and let $U_1$ and $U_2$ be the regular and trivial representation of $\GG$. Then $U=U_1\oplus{U_2}\in\M\bigl(\K(\ell^2(\ZZ)\oplus\CC)\tens\cc_0(\ZZ)\bigr)$ generates $\cc_0(\ZZ)$. It cannot be quasi-equivalent to $\ww^{\ZZ}$, as then, according to Theorem \ref{QE = Fourier}, we would have a (normal) $*$-isomorphism between von Neumann algebras $\sA_U''=L^{\infty}(\TT)\oplus\CC$ and $\sA_{\ww^\ZZ}''=L^{\infty}(\TT)$. However, the latter algebra is non-atomic, so we would have a contradiction.

\end{proof}

\section{Closed quantum subgroups of locally compact quantum groups}\label{Subgroups}

This section is central to our paper. We begin by introducing two possible definitions of a closed quantum subgroup of a given quantum group, the first of which appears in \cite{Vaes} and the second was suggested to us by S.L.~Woronowicz. Then we provide alternative, simplified descriptions for both of them (Theorems \ref{Vaes-equiv} and \ref{Woronequiv}) and analyze their mutual relations (Theorems \ref{Vaes->Wor} and \ref{Vaesfourier}). We also present here a proof of Theorem \ref{isoT}.

The aforementioned definitions are as follows:

\begin{definition}[{\cite[Definition 2.5]{Vaes}}] \label{Vsubgroup}
Let $\GG$, $\HH$ be locally compact quantum groups. Then $\HH$ is said to be a \emph{closed quantum subgroup of $\GG$ in the sense of Vaes} if there exists a morphism $\pi$ from $\HH$ to $\GG$ and a normal injective $*$-homomorphism $\gamma\colon{L^{\infty}(\hh{\HH})}\to{L^{\infty}(\hh{\GG})}$ such that
\begin{equation}\label{Vaesdual}
\bigl.\gamma\bigr|_{\C_0(\hh{\HH})}\comp\Lambda_{\hh{\HH}}=\Lambda_{\hh{\GG}}\comp\hh{\pi}.
\end{equation}
\end{definition}

\begin{definition}[Woronowicz] \label{Wsubgroup}
Let $\GG$, $\HH$ be locally compact quantum groups. Then $\HH$ is said to be a \emph{closed quantum subgroup of $\GG$ in the sense of Woronowicz} if there exists a morphism $\pi$ from $\HH$ to $\GG$ such that the associated bicharacter $V\in\M\bigl(\C_0(\hh{\GG})\tens\C_0(\HH)\bigr)$ generates $\C_0(\HH)$.
\end{definition}

The conditions above take as a starting point a morphism $\pi$ from $\HH$ to $\GG$. We will sometimes say that $\HH$ is a closed quantum subgroup of $\GG$ in the sense of Vaes (respectively, in the sense of Woronowicz) \emph{via the morphism $\pi$}. In Section \ref{commut} we will explain why when both $\HH$ and $\GG$ are locally compact groups both definitions are equivalent to the classical notion of $\HH$ being (homeomorphic to) a closed subgroup of $\GG$.

We will see later that the various examples of quantum subgroups considered in the literature are all closed quantum subgroups in the sense of both Vaes and Woronowicz. The case of compact  and discrete subgroups is treated in Section \ref{cptdiscr}. The non-compact examples of quantum subgroups in \cite[Sections 3 and 4]{exqga} and those coming from Rieffel deformation presented in \cite{kasp} are all closed subgroups in the sense of Vaes and Woronowicz (\cite[Section 6]{kasp}, cf.~also Theorem \ref{Vaes->Wor}). Another class of examples is provided by the bicrossed product construction (see e.g.~\cite{VV,VV2}). If $(\GG_1,\GG_2)$ is a matched pair of locally compact quantum groups in the sense of \cite[Definition 2.1]{VV2} then $\hh{\GG}_1$ is a closed quantum subgroup of the bicrossed product of $\GG_1$ and $\GG_2$ both in the sense of Vaes and Woronowicz.

In the next theorem we note that the definition of Vaes can be reformulated in various simplified ways (note especially condition \eqref{VW12}, which does not assume a priori the existence of a homomorphism between $\HH$ and $\GG$). In particular the definition of Vaes-Vainerman (\cite[Definition 2.9]{VV}) is equivalent to Definition \ref{Vsubgroup}.

\begin{theorem}\label{VW1}
Let $\GG$, $\HH$ be locally compact quantum groups. Then the following conditions are equivalent:
\begin{enumerate}
\item\label{VW11} $\HH$ is a closed quantum subgroup of $\GG$ in the sense of Vaes;
\item\label{VW12} there exists a normal injective $*$-homomorphism
$\gamma\colon{L^{\infty}(\hh{\HH})}\to{L^{\infty}(\hh{\GG})}$ such that
\begin{equation}\label{intr}
(\gamma\tens\gamma)\comp\Delta_{\hh{\HH}}=\Delta_{\hh{\GG}}\comp\gamma;
\end{equation}
\item\label{VW13} there exists a normal injective $*$-homomorphism
$\gamma\colon{L^{\infty}(\hh{\HH})}\to{L^{\infty}(\hh{\GG})}$ such that the unitary
$(\gamma\tens\id)(\ww^{\HH})\in{L^{\infty}(\hh{\GG})}\vtens{L^{\infty}(\HH)}$ is a bicharacter from $\HH$ to $\GG$ --- in particular it belongs to $\M\bigl(\C_0(\hh{\GG})\tens\C_0(\HH)\bigr)$.
\end{enumerate}
\end{theorem}

It will become clear from the proof of Theorem \ref{VW1} that the map $\gamma$ mentioned in point \eqref{VW12} is the same as the one in \eqref{VW13} and still the same as the map $\gamma$ from Definition \ref{Vsubgroup}. Moreover we show in the proof that $\gamma$ restricted to $\C_0(\hh{\HH})$ is an element of $\Mor\bigl(\C_0(\hh{\HH}),\C_0(\hh{\GG})\bigr)$.

\begin{proof}[Proof of Theorem \ref{VW1}]
\eqref{VW11}$\Rightarrow$\eqref{VW12} --- trivial.

\eqref{VW12}$\Rightarrow$\eqref{VW13}. The map $\gamma'=\bigl.\gamma\bigr|_{\C_0(\hh{\HH})}$ is naturally a representation of the \cst-algebra $\C_0(\hh{\HH})$ on $L^2(\GG)$. Consider the unitary $V=(\gamma'\tens\id)(\ww^{\HH})\in\M\bigl(\K\bigl(L^2(\GG)\bigr)\tens\C_0(\HH)\bigr)$. Applying $\gamma\tens\gamma\tens\id$ to both sides of the equality
\begin{equation*}
(\Delta_{\hh{\HH}}\tens\id)(\ww^{\HH})=\ww^{\HH}_{23}\ww^{\HH}_{13}
\end{equation*}
(viewed as an equality of operators in $L^{\infty}(\hh{\HH})\,\vtens\,{L^{\infty}(\hh{\HH})}\,\vtens\,{L^{\infty}(\HH)}$) and using the equation \eqref{intr} we see that
\begin{equation}\label{exactly}
V_{13}=(\ww^{\GG}_{12})^*V_{23}\ww^{\GG}_{12}V_{23}^*.
\end{equation}
The right side of the above expression belongs to
$\M\bigl(\C_0(\hh{\GG})\tens\K\bigl(L^2(\GG)\bigr)\tens\C_0(\HH)\bigr)$ and $V_{13}$ has legs only in the first and third tensor factor. Thus $V\in\M\bigl(\C_0(\hh{\GG})\tens\C_0(\HH)\bigr)$. Note that \eqref{exactly} is precisely \eqref{VG}.
The application of $\gamma\tens\id\tens\id$ to the pentagonal equation for $\ww^{\HH}$ implies that \eqref{VH} holds and $V$ is a bicharacter.

\eqref{VW13}$\Rightarrow$\eqref{VW11}. Note first that as $V\in\M\bigl(\C_0(\hh{\GG})\tens\C_0(\HH)\bigr)$, it follows that
\begin{equation*}
\gamma'=\bigl.\gamma\bigr|_{\C_0(\hh\HH)}\in\Mor\bigl(\C_0(\hh{\HH}),\C_0(\hh{\GG})\bigr)
\end{equation*}
because $\ww^{\HH}\in\M\bigl(\C_0(\hh{\HH})\tens\C_0(\HH)\bigr)$ generates $\C_0(\hh{\HH})$.

Let $\pi\in\Mor\bigl(\C_0^\uu(\GG),\C_0^\uu(\HH)\bigr)$ be the strong quantum homomorphism associated with the bicharacter $V$ and let $\hh{\pi}\in\Mor\bigl(\C_0^\uu(\hh{\HH}),\C_0(\hh{\GG})\bigr)$ be the dual quantum homomorphism. Then on one hand we have (recall the dependencies between $V$, $\hh{V}$ and $\hh{\pi}$ listed in Subsection \ref{morphism})
\begin{equation*}
(\id\tens\gamma')(\ww^{\hh{\HH}})=\hh{V}
\end{equation*}
and on the other hand
\begin{equation*}
\hh{V}=\bigl(\Lambda_{\HH}\tens[\Lambda_{\hh{\GG}}\comp\hh{\pi}]\bigr)(\WW^{\hh{\HH}}).
\end{equation*}
(this is \eqref{this} combined with \eqref{DualBigW}). Comparing the above and using the fact that
\begin{equation*}
\ww^{\hh{\HH}}=(\Lambda_{\HH}\tens\Lambda_{\hh{\HH}})(\WW^{\hh{\HH}})
\end{equation*}
(i.e.~\eqref{down} for the quantum group $\hh{\HH}$) we obtain
\begin{equation*}
\bigl(\Lambda_{\HH}\tens[\gamma'\comp\Lambda_{\hh{\HH}}]\bigr)(\WW^{\hh{\HH}})
=\bigl(\Lambda_{\HH}\tens[\Lambda_{\hh{\GG}}\comp\hh{\pi}]\bigr)(\WW^{\hh{\HH}}).
\end{equation*}
This can be rewritten as
\begin{equation*}
\bigl(\id\tens[\gamma'\comp\Lambda_{\hh{\HH}}]\bigr)(\wW^{\hh{\HH}})
=\bigl(\id\tens[\Lambda_{\hh{\GG}}\comp\hh{\pi}]\bigr)(\wW^{\hh{\HH}}).
\end{equation*}
and upon application of $(\omega\tens\id)$ with $\omega\in\B\bigl(L^2(\HH)\bigr)_*$ yields
\begin{equation*}
(\gamma'\comp\Lambda_{\hh{\HH}})\bigl((\omega\tens\id)(\wW^{\hh{\HH}})\bigr)=
(\Lambda_{\hh{\GG}}\comp\hh{\pi})\bigl((\omega\tens\id)(\wW^{\hh{\HH}})\bigr)
\end{equation*}
for any such $\omega$. By \eqref{C0u} this implies that \eqref{Vaesdual} holds. This ends the proof.
\end{proof}

Theorem \ref{VW1} and a straightforward application of Theorem \ref{QE = Fourier} yields the following result.

\begin{theorem}\label{Vaes-equiv}
Let $\GG$, $\HH$ be locally compact quantum groups and suppose that $V\in\M\bigl(\C_0(\hh{\GG})\tens\C_0(\HH)\bigr)$ is a bicharacter describing a morphism $\pi$ from $\HH$ to $\GG$. Then the following conditions are equivalent:
\begin{enumerate}
\item\label{Vaes-equiv1} $\HH$ is a closed quantum subgroup of $\GG$ in the sense of Vaes via the morphism $\pi$;
\item\label{Vaes-equiv2} the bicharacter $V$ is quasi-equivalent to $\ww^{\HH}$ (as a representation of $\HH$);
\item\label{Vaes-equiv3} $\cA_{\hh{V}}=\cA_{\ww^{\hh{\HH}}}$.
\end{enumerate}
\end{theorem}

\begin{proof}
The equivalence of \eqref{Vaes-equiv2} and \eqref{Vaes-equiv3} follows immediately from Theorem \ref{QE = Fourier}, as the sets appearing in \eqref{Vaes-equiv3} here coincide with the analogous sets in condition \eqref{QEF3} of that theorem (recall that $\hh{V}=\sigma(V^*)$, $\ww^{\hh{\HH}}=\sigma(\ww^{\HH})^*$). The equivalence of \eqref{Vaes-equiv1} and \eqref{Vaes-equiv2} follows again from Theorem \ref{QE = Fourier} and (the proof of) Theorem \ref{VW1}.
\end{proof}

It now follows from Theorem \ref{Vaes-equiv} and Proposition \ref{QE -> generating}  that there is a natural relation between Definitions \ref{Vsubgroup} and \ref{Wsubgroup}.

\begin{theorem} \label{Vaes->Wor}
If $\HH$ is a closed quantum subgroup of $\GG$ in the sense of Vaes, it is also a closed quantum subgroup of $\GG$ in the sense of Woronowicz.
\end{theorem}

\begin{proof}
Immediate consequence of Theorem \ref{Vaes-equiv} and Proposition \ref{QE -> generating}.
\end{proof}

It is not clear if Definitions \ref{Vsubgroup} and \ref{Wsubgroup} are equivalent; in other words, whether Theorem \ref{Vaes->Wor} admits the converse. This would follow if we could show that a bi-character $U\in\M\bigl(\C_0(\hh{\GG})\tens\C_0(\HH)\bigr)$ describing a homomorphism from $\HH$ to $\GG$ which generates $\C_0(\HH)$ must be quasi-equivalent to $\ww^{\HH}$ (the example in Proposition \ref{QE -> generating} showed it need not be the case if we only assume that $U$ is a representation of $\HH$). In the following sections we will show that in fact the equivalence of Definitions \ref{Vsubgroup} and \ref{Wsubgroup} holds in many natural cases.

Now we show that Definition \ref{Wsubgroup} also admits several natural equivalent reformulations. We collect them in the next theorem.

\begin{theorem}\label{Woronequiv}
Let $\GG$, $\HH$ be locally compact quantum groups and consider a homomorphism from $\HH$ to $\GG$ described by a bicharacter $V\in\M\bigl(\C_0(\hh{\GG})\tens\C_0(\HH)\bigr)$, a strong quantum homomorphism $\pi\in\Mor\bigl(\C_0^\uu(\GG),\C_0^\uu(\HH)\bigr)$ and a right quantum homomorphism $\rho\in\Mor\bigl(\C_0(\GG),\C_0(\GG)\tens\C_0(\HH)\bigr)$. Then the following conditions are equivalent (recall that $\hh{V}:=\sigma(V)^*$ is a representation of $\hh{\GG}$):
\begin{enumerate}
\item\label{Woronequiv1} $V\in\M\bigl(\C_0(\hh{\GG})\tens\C_0(\HH)\bigr)$ generates $\C_0(\HH)$ (in other words $\HH$ is a closed quantum subgroup of $\GG$ in the sense of Woronowicz);
\item\label{Woronequiv2} $\sA_{\hh{V}}=\C_0(\HH)$;
\item\label{Woronequiv3} the right quantum homomorphism $\rho$ is \emph{strongly non-degenerate}:
\begin{equation}\label{rightqh}
\bb{\rho\bigl(\C_0(\GG)\bigr)\bigl(\C_0(\GG)\tens\I_{\C_0(\HH)}\bigr)}=\C_0(\GG)\tens\C_0(\HH)
\end{equation}
(in particular the left hand side of \eqref{rightqh} is contained in the right hand side);
\item\label{Woronequiv4} $\pi\bigl(\C_0^\uu(\GG)\bigr)=\C_0^\uu(\HH)$;
\item\label{Woronequiv5} $(\Lambda_{\HH}\comp\pi)\bigl(\C_0^\uu(\GG)\bigr)=\C_0(\HH)$.
\end{enumerate}
\end{theorem}

\begin{proof}
\eqref{Woronequiv1}$\Leftrightarrow$\eqref{Woronequiv2}. This follows from Proposition \ref{generatingAU} and an obvious fact that $V\in\M\bigl(\C_0(\hh{\GG})\tens\C_0(\HH)\bigr)$ generates $\C_0(\HH)$ if and only if $\hh{V}\in\M\bigl(\C_0(\HH)\tens\C_0(\hh{\GG})\bigr)$ generates $\C_0(\HH)$ (cf.~Remark \ref{order}).

\eqref{Woronequiv2}$\Leftrightarrow$\eqref{Woronequiv3}. We compute:
\begin{equation*}
\begin{split}
&\!\!\bb{\rho\bigl(\C_0(\GG)\bigr)\bigl(\C_0(\GG)\tens\I\bigr)}\\
&=\bigl\{V\bigl((x\cdot\omega\tens\id)(\ww^{\GG})^*\tens\I\bigr)V^*(y\tens\I)
\St{x}\in\C_0(\hh{\GG}),\:y\in\C_0(\GG),\:\omega\in\B\bigl(L^2(\GG)\bigr)_*\bigr\}^\nc\\
&=\bigl\{(\omega\tens\id\tens\id)\bigl(V_{23}(\ww^{\GG}_{12})^*V_{23}^*(x\tens{y}\tens\I)\bigr)
\St{x}\in\C_0(\hh{\GG}),\:y\in\C_0(\GG),\:\omega\in\B\bigl(L^2(\GG)\bigr)_*\bigr\}^\nc\\
&=\bigl\{(\omega\tens\id\tens\id)\bigl(V_{13}^*(\ww^{\GG}_{12})^*(x\tens{y}\tens\I)\bigr)
\St{x}\in\C_0(\hh{\GG}),\:y\in\C_0(\GG),\:\omega\in\B\bigl(L^2(\GG)\bigr)_*\bigr\}^\nc\\
&=\bigl\{(\omega\tens\id\tens\id)\bigl(V_{13}^*(x\tens{y}\tens\I)\bigr)
\St{x}\in\C_0(\hh{\GG}),\:y\in\C_0(\GG),\:\omega\in\B\bigl(L^2(\GG)\bigr)_*\bigr\}^\nc\\
&=\C_0(\GG)\tens\bigl\{(\omega\tens\id)(V)\St\omega\in\B\bigl(L^2(\GG)\bigr)_*\bigr\}^\nc
=\C_0(\GG)\tens\sA_{\hh{V}}
\end{split}
\end{equation*}
In the third equality we used the bicharacter property of $V$ (Eq.~\eqref{VG}) and in the fourth equality we used the fact that $\ww^{\GG}\in\M\bigl(\C_0(\hh{\GG})\tens\C_0(\GG)\bigr)$ is unitary. The above computation shows that $\sA_{\hh{V}}=\C_0(\HH)$ if and only if $\bb{\rho\bigl(\C_0(\GG)\bigr)\bigl(\C_0(\GG)\tens\I\bigr)}=\C_0(\GG)\tens\C_0(\HH)$.

\eqref{Woronequiv2}$\Leftrightarrow$\eqref{Woronequiv5}.
Taking into account \eqref{C0u}, \eqref{wWWw} and \eqref{Vdef} we find that $V=\bigl(\id\tens[\Lambda_{\HH}\comp\pi]\bigr)(\wW^{\GG})$ and
\begin{equation}\label{Lampi}
(\Lambda_{\HH}\comp\pi)\bigl((\omega\tens\id)(\wW^{\GG})\bigr)=(\omega\tens\id)(V).
\end{equation}
Thus
\begin{equation*}
\begin{split}
(\Lambda_{\HH}\comp\pi)\bigl(\C_0^\uu(\GG)\bigr)&=(\Lambda_{\HH}\comp\pi)
\bigl(\bigl\{(\omega\tens\id)(\wW^{\GG})\St\omega\in\B\bigr(L^2(\GG)\bigr)_*\bigr\}^\nc\bigr)\\
&=\bigl\{(\Lambda_{\HH}\comp\pi)\bigl((\omega\tens\id)(\wW^{\GG})\bigr)
\St\omega\in\B\bigr(L^2(\GG)\bigr)_*\bigr\}^\nc=\sA_{\hh{V}}.
\end{split}
\end{equation*}

\eqref{Woronequiv4}$\Leftrightarrow$\eqref{Woronequiv5}.
Since \eqref{Woronequiv4}$\Rightarrow$\eqref{Woronequiv5} is clear, it remains to show the converse implication. Consider the universal lift of $V^\uu\in\M\bigl(\C_0(\hh{\GG})\tens\C_0^\uu(\HH)\bigr)$ defined as $V^\uu=(\id\tens\pi)(\wW^{\GG})$, cf.~\cite[Section 4]{MRW}. To show the desired implication it suffices to establish the following equality:
\begin{equation*}
\bigl\{(\omega\tens\id)(V^\uu)\St\omega\in\B\bigl(L^2(\GG)\bigr)_*\bigr\}^\nc=\C_0^\uu(\HH)
\end{equation*}
(cf.~the proof of \eqref{Woronequiv2}$\Leftrightarrow$\eqref{Woronequiv5}). Noting that\footnote{\label{foot}To prove \eqref{VVWVW} we first note that we have $(\Delta_{\hh{\HH}}\tens\id)\wW^{\HH}=\wW^{\HH}_{13}\wW^{\HH}_{23}$ as $\wW^\HH$ is the unique lift of $\ww^{\HH}$ to a bicharacter in $\M\bigl(\C_0(\hh{\HH})\tens\C_0^\uu(\HH)\bigr)$ described in \cite[Proposition 4.14]{MRW}. This can be rewritten as $\ww^{\HH}_{12}\wW^{\HH}_{13}=\wW^{\HH}_{23}\ww^{\HH}_{12}(\wW^{\HH}_{23})^*$. Slicing with $\omega\in\B\bigl(L^2(\HH)\bigr)_*$ on the left leg we obtain the formula $(\Lambda_{\HH}\tens\id)\Delta_{\HH}^\uu(x)=\wW^\HH\bigl(\Lambda_\HH(x)\tens\I\bigr)(\wW^\HH)^*$ for all $x\in\C_0^\uu(\HH)$. Now we apply $\id\tens\Lambda_{\HH}\tens\id$ to both sides of $(\id\tens\Delta_{\HH}^\uu)(V^\uu)=V^\uu_{12}V^\uu_{13}$. By the previous formula this reads $\wW^\HH_{23}V_{12}(\wW^\HH_{23})^*=V_{12}V^\uu_{13}$ which is \eqref{VVWVW}.}
\begin{equation}\label{VVWVW}
V_{13}^\uu=V_{12}^*\wW^{\HH}_{23}V_{12}(\wW^{\HH}_{23})^*
\end{equation}
we compute:
\begin{equation*}
\begin{split}
\bigl\{(\omega\tens\id)(V^\uu)&\St\omega\in\B\bigl(L^2(\GG)\bigr)_*\bigr\}^\nc\\
&=\bigl\{(\omega\tens\mu\tens\id)\bigl(V_{12}^*\wW^{\HH}_{23}V_{12}(\wW^{\HH}_{23})^*\bigr)
\St\omega\in\B\bigl(L^2(\GG)\bigr)_*,\:\mu\in\B\bigl(L^2(\HH)\bigr)_*\bigr\}^\nc\\
&=\bigl\{(\omega\tens\mu\tens\id)\bigl(\wW^\HH_{23}V_{12}(\wW^\HH_{23})^*\bigr)
\St\omega\in\B\bigl(L^2(\GG)\bigr)_*,\:\mu\in\B\bigl(L^2(\HH)\bigr)_*\bigr\}^\nc\\
&=\bigl\{
(\eta\tens\mu\tens\id)\bigl(\wW^\HH_{23}\ww^{\HH}_{12}(\wW^\HH_{23})^*\bigr)
\St\eta\in\B\bigl(L^2(\HH)\bigr)_*,\:\mu\in\B\bigl(L^2(\HH)\bigr)_*\bigr\}^\nc\\
&=\bigl\{(\eta\tens\mu\tens\id)(\ww^{\HH}_{12}\wW^\HH_{13})
\St\eta\in\B\bigl(L^2(\HH)\bigr)_*,\:\mu\in\B\bigl(L^2(\HH)\bigr)_*\bigr\}^\nc\\
&=\bigl\{(\eta\cdot{x}\tens\id)(\wW^\HH_{13})\St\eta\in\B\bigl(L^2(\HH)\bigr)_*,\:
{x}\in\C_0(\hh{\HH})\bigr\}^\nc=\C_0^\uu(\HH).
\end{split}
\end{equation*}
The second equality follows from the unitarity of $V$, in the third one we used the fact that $V\in\M\bigl(\C_0(\hh{\GG})\tens\C_0(\HH)\bigr)$ generates $\C_0(\HH)$ and in the fourth equality we
used the equation $\wW^\HH_{23}\ww^{\HH}_{12}(\wW^\HH_{23})^*=\ww^{\HH}_{12}\wW^\HH_{13}$ (see derivation of formula \eqref{VVWVW}).
\end{proof}

Condition \eqref{Woronequiv3} in Theorem \ref{Woronequiv} classically corresponds to properness and freeness of the natural action of $\HH$ on $\GG$ induced by the homomorphism from $\HH$ to $\GG$, see Section \ref{commut}. It was introduced in the context of quantum groups by Podle\'s in his thesis \cite[Definicja 2.2]{podlesPhD}, see also \cite{podles} and \cite[Proposition 2.3]{exqga} for a complete discussion. Condition \eqref{Woronequiv4} is a natural reflection of a general principle that injectivity on the level of point transformations is equivalent to surjectivity on the level of induced transformations on algebras of functions (but cf.~Theorem \ref{std}\eqref{std4}). One can ask the following question: are the conditions \eqref{Woronequiv1}-\eqref{Woronequiv3} above equivalent in general to the surjectivity of the extension of $\pi$ to multiplier algebras $\bar{\pi}\colon\M\bigl(\C_0^\uu(\GG)\bigr)\to\M\bigl(\C_0^\uu(\HH)\bigr)$? The Pedersen-Tietze theorem
(\cite[Theorem 2.3.9]{W-O}) implies that $\bar{\pi}$ is surjective if $\pi$ is, provided that $\C_0^\uu(\GG)$ (and hence $\C_0^\uu(\HH)$) is $\sigma$-unital.

With the results obtained in this section in hand we are ready to give a proof of Theorem \ref{isoT}.

\begin{proof}[Proof of Theorem \ref{isoT}]
\sloppy
We are assuming that the strong quantum homomorphism $\pi\in\Mor\bigl(\C_0^\uu(\GG),\C_0^\uu(\HH)\bigr)$ is an isomorphism. Consider the dual strong quantum homomorphism $\hh{\pi}\in\Mor\bigl(\C_0^\uu(\hh{\HH}),\C_0^\uu(\hh{\GG})\bigr)$. By \eqref{pipi} $\hh{\pi}$ is an isomorphism. We will show that $\hh{\pi}$ identifies $\hh{\GG}$ with a Vaes-closed subgroup of $\hh{\HH}$. Let $\hh{V}$ be the bicharacter associated to $\hh{\pi}$:
\begin{equation*}
\hh{V}=\bigl(\Lambda_{\HH}\tens[\Lambda_{\hh{\GG}}\comp\hh{\pi}]\bigr)(\WW^{\hh{\HH}}).
\end{equation*}
As
\begin{equation*}
\wW^{\hh{\HH}}\bigl(\Lambda_{\hh{\HH}}(x)\tens\I\bigr)(\wW^{\hh{\HH}})^*=
(\Lambda_{\hh{\HH}}\tens\id)\Delta_{\hh{\HH}}^\uu(x)
\end{equation*}
for all $x\in\C_0^\uu(\hh{\HH})$ (see Footnote \ref{foot}), it follows that
\begin{equation*}
\hh{V}\bigl(\Lambda_{\hh{\HH}}(x)\tens\I\bigr)(\hh{V})^*
=\bigl(\Lambda_{\hh{\HH}}\tens[\Lambda_{\hh{\GG}}\comp\hh{\pi}]\bigr)\Delta_{\hh{\HH}}^\uu(x).
\end{equation*}
Let $Y\in\M\bigl(\K(\sH_Y)\tens\C_0(\HH)\bigr)$ be a representation of $\HH$. By the results of \cite{Johanuniv,MRW}, there exists a unique $\varphi\in\Mor\bigl(\C_0^\uu(\HH),\K(\sH)\bigr)$ such that $(\varphi\tens\Lambda_{\hh{\HH}})(\WW^{\hh{\HH}})=Y$. Set $X=(\varphi\tens\hh{\pi}^{-1})(\WW^{\hh{\HH}})$. Then $(\id\tens\Delta_{\hh{\HH}}^\uu)(X)=X_{12}X_{13}$ because $\hh{\pi}^{-1}$ intertwines the coproducts. Further
\begin{equation*}
\bigl(\id\tens\Lambda_{\hh{\HH}}\tens[\Lambda_{\hh{\GG}}\comp\pi]\bigr)(X_{12}X_{13})
=\hh{V}_{23}\bigl((\id\tens\Lambda_{\hh{\HH}})(X)\bigr)_{12}(\hh{V}_{23})^*,
\end{equation*}
so
\begin{equation*}
\bigl(\id\tens[\Lambda_{\hh{\GG}}\comp\pi]\bigr)(X)\tp{\hh{V}}=
\bigl((\id\tens\Lambda_{\hh{\HH}})(X)\bigr)_{12}^*
\hh{V}_{23}
\bigl((\id\tens\Lambda_{\hh{\HH}})(X)\bigr)_{12}
\end{equation*}
which means that $\bigl(\id\tens[\Lambda_{\hh{\GG}}\comp\pi]\bigr)(X)\tp{\hh{V}}$ is equivalent to $\II_{\sH_Y}\tp\hh{V}$. However $\bigl(\id\tens[\Lambda_{\hh{\GG}}\comp\pi]\bigr)(X)=Y$, so that, as $Y$ was arbitrary, $\hh{V}$ is right-absorbing. It follows from \cite{WorSol} (see remark before Proposition \ref{QE -> generating}) that $\hh{V}$ is quasi-equivalent to $\ww^{\hh{\HH}}$. By Theorem \ref{Vaes-equiv} $\hh{\GG}$ is a closed quantum subgroup of $\hh{\HH}$ in the sense of Vaes.

By Theorem \ref{VW1} and the comment after it, there exists a morphism $\hh{\gamma}_1\in\Mor\bigl(\C_0(\GG),\C_0(\HH)\bigr)$ such that
\begin{equation*}
\hh{\gamma}_1\comp\Lambda_{\GG}=\Lambda_{\HH}\comp\pi.
\end{equation*}
Applying identical reasoning to $\hh{\pi}^{-1}$ we obtain the existence of $\hh{\gamma}_2\in\Mor\bigl(\C_0(\HH),\C_0(\GG)\bigr)$ such that
\begin{equation*}
\hh{\gamma}_2\comp\Lambda_{\HH}=\Lambda_{\GG}\comp{\pi}^{-1}
\end{equation*}
(note that we use once again the fact that $\hh{\pi}^{-1}=\hh{\pi^{-1}}$). Since $\Lambda_{\GG}$ and $\Lambda_\HH$ are surjections we see that $\hh{\gamma}_1$ and $\hh{\gamma}_2$ are mutually inverse and we can set $\pir=\hh{\gamma}_1$.
\end{proof}

We recall from Subsection \ref{QGsub} that the Fourier algebra and the Fourier-Stieltjes algebra of a locally compact quantum group are the Banach spaces
\begin{equation*}
\cA_\GG=L^\infty(\hh{\GG})_*\qquad\text{and}\qquad\cB_\GG=\C_0^\uu(\hh{\GG})^*
\end{equation*}
which we embedded into $\C_0(\GG)$ and  $\M\bigl(\C_0^\uu(\GG)\bigr)$ respectively with the maps
\begin{equation*}
\begin{split}
\cA_\GG\ni\omega&\longmapsto(\omega\tens\id)(\ww^{\GG}),\\
\cB_\GG\ni\eta&\longmapsto(\eta\tens\id)(\WW^{\GG}).
\end{split}
\end{equation*}
We also note that $\cA_\GG$ embeds in $\cB_\GG$ via $\Lambda_{\hh{\GG}}^*\comp\imath^*$, where $\imath$ is the embedding $\C_0(\GG)\hookrightarrow{L^\infty(\GG)}$ (we will also use the symbol ``$\imath$'' to denote the analogous embedding for other quantum groups). It is easy to check that this embedding is isometric. Moreover the induced embedding of $\cA_\GG$ into $\M\bigl(\C_0^\uu(\GG)\bigr)$ actually embeds the Fourier algebra into $\C_0^\uu(\GG)$.

\begin{theorem}\label{Vaesfourier}
Let $\HH$ be a closed subgroup of $\GG$ in the sense of Woronowicz via the morphism $\pi\colon\C_0^\uu(\GG)\to\C_0^\uu(\HH)$. Then the following are equivalent:
\begin{enumerate}
\item\label{vaesfourier1} $\pi$ restricts to a map $T\colon\cA_\GG\to\cA_\HH$ which has dense range (for the $\cA_\HH$ norm);
\item\label{vaesfourier2} $\hh{\pi}^*\colon\cB_\GG\to\cB_\HH$ restricts to a map $S\colon\cA_\GG\to\cA_\HH$ which has dense range;
\item\label{vaesfourier3} $\HH$ is a closed subgroup of $\GG$ in the sense of Vaes.
\end{enumerate}
Moreover we can replace ``dense range'' by ``surjection'' in \eqref{vaesfourier1} and \eqref{vaesfourier2}.  If these conditions hold, then $S$ and $T$ are the same map, which is nothing but the pre-adjoint of the implicit map $\gamma\colon{L^\infty(\hh{\HH})}\to{L^\infty(\hh{\GG})}$ appearing in \eqref{vaesfourier3}.
\end{theorem}

\begin{proof}
Let $\omega\in{L^\infty(\hh\GG)_*}=\cA_\GG$, set $\mu=\Lambda_{\hh\GG}^*\bigl(\imath^*(\omega)\bigr)\in\C_0^\uu(\hh\GG)^*$, and let $a$ be the image of $\omega$ in $\cB_\GG\subset\M\bigl(\C_0^\uu(\GG)\bigr)$, so $a=(\mu\tens\id)(\WW^\GG)$. Then
\begin{equation*}
\pi(a)=(\mu\tens\pi)(\WW^\GG)=\bigl(\hh{\pi}^*(\mu)\tens\id\bigr)(\WW^\HH),
\end{equation*}
so that $\pi(a)$ is (the image of) $\hh{\pi}^*(\mu)$ in $\cB_\HH$. It is now clear that \eqref{vaesfourier1} and \eqref{vaesfourier2} are equivalent.

If \eqref{vaesfourier2} holds then the map $S$ satisfies $\Lambda_{\hh\HH}^*\comp\imath^*\comp{S}=\hh{\pi}^*\comp\Lambda_{\hh\GG}^*\comp\imath^*$, and as $\Lambda_{\hh\GG}^*\comp\imath^*$ and $\Lambda_{\hh\HH}^*\comp\imath^*$ are isometries, $S$ must be bounded.  Set $\gamma=S^*\colon{L^\infty(\hh{\HH})}\to{L^\infty(\hh{\GG})}$, so as $S$ has dense range, $\gamma$ is injective. Then we have that
\begin{equation*}
\gamma\comp\Lambda_{\hh\HH}=S^*\comp\Lambda_{\hh\HH}
=S^*\comp\imath^{**}\comp\bigl.\Lambda_{\hh\HH}^{**}\bigr|_{\C_0^\uu(\hh\HH)}
=\imath^{**}\comp\Lambda_{\hh\GG}^{**}\comp\bigl.\hh{\pi}^{**}\bigr|_{\C_0^\uu(\hh\HH)}
=\imath\comp\Lambda_{\hh\GG}\comp\hh{\pi}.
\end{equation*}
As $\gamma$ is weak$^*$-continuous, it now follows that $\gamma$ is a $*$-homomorphism, and so \eqref{vaesfourier3} holds.

Finally, if \eqref{vaesfourier3} holds, then we have a normal injective $*$-homomorphism $\gamma\colon{L^\infty(\hh{\HH})}\to{L^\infty(\hh{\GG})}$ with $\gamma\comp\imath\comp\Lambda_{\hh\HH}=\imath\comp\Lambda_{\hh\GG}\comp\hh{\pi}$. Thus, for $\omega\in{L^\infty(\hh\GG)_*}$, we have that
\begin{equation*}
\hh{\pi}^*\bigl((\Lambda_{\hh\GG}^*\comp\imath^*)(\omega)\bigr)=
(\Lambda_{\hh\HH}^*\comp\imath^*\comp\gamma_*)(\omega),
\end{equation*}
and so $\hh{\pi}^*$ restricts to a map $S\colon\cA_\GG\to\cA_\HH$. As $\gamma$ is injective and hence an isometry, $\gamma_*$ is a surjection and so $S$, which agrees with $\gamma_*$ once appropriate identifications are made, is also a surjection. This shows \eqref{vaesfourier2}, and also demonstrates the claim about replacing ``dense range'' by ``surjection''.
\end{proof}

\begin{remark}\label{Vf}
From Theorem \ref{Vaesfourier} we immediately see that $\HH$ is a closed subgroup of $\GG$ in the sense of Vaes, via the morphism $\pi\colon\C_0^\uu(\GG)\to\C_0^\uu(\HH)$, if and only if $\pi$ restricts to a surjection $\cA_\GG\to\cA_\HH$. In the classical case, the Herz restriction theorem (\cite{herz,arsac}) says exactly that if $H$ is a closed subgroup of $G$, then the restriction map (which is nothing but $\pi\colon\C_0(G)\to\C_0(H)$) gives a surjection $\cA_G\to\cA_H$. In other words the definition of a Vaes-closed subgroup is tailored exactly so that the quantum version of the Herz restriction theorem holds.
\end{remark}

\section{Commutative case}\label{commut}

Let now $G$ and $H$ be locally compact groups, so in particular $\C_0^\uu(G)=\C_0(G)$ and $\C_0^\uu(H)=\C_0(H)$. Any homomorphism from $H$ to $G$ (in the sense of quantum groups --- as defined in Subsection \ref{morphism}) is then described by a $\pi\in\Mor\bigl(\C_0(G),\C_0(H)\bigr)$. Moreover $\pi$ is necessarily of the form $\pi(f)=f\comp\theta$, where $\theta\colon{H}\to{G}$ is a continuous homomorphism (cf.~Theorem \ref{basic}).

Given a situation as above, consider the natural right action of $H$ on the topological space $G$ given by
\begin{equation}\label{pre_rho}
G\times{H}\ni(g,h)\longmapsto{g\cdot{h}}=g\,\theta(h)\in{G}.
\end{equation}
Let us also introduce the so called \emph{canonical map} $\gamma\colon{G\times{H}}\to{G}\times{G}$ for this action
\begin{equation}\label{Gamma}
\gamma(g,h)=(g,g\cdot{h})=\bigl(g,g\,\theta(h)\bigr).
\end{equation}
(\cite{schneider}). Let $\rho\in\Mor\bigl(\C_0(G),\C_0(G)\tens\C_0(H)\bigr)$ and $\Gamma\in\Mor\bigl(\C_0(G)\tens\C_0(G),\C_0(G)\tens\C_0(H)\bigr)$ be the morphisms of \cst-algebras corresponding to \eqref{pre_rho} and \eqref{Gamma}:
\begin{equation*}
\begin{aligned}
\rho(f)(g,h)&=f(g\cdot{h}),&f\in\C_0(G),&&g\in{G},\:h\in{H},\\
\Gamma(F)(g,h)&=F\bigl(\gamma(g,h)\bigr),&F\in\C_0(G)\tens\C_0(G),&&g\in{G},\:h\in{H}.
\end{aligned}
\end{equation*}

\begin{lemma}\label{clas}
Let $\theta\colon{H}\to{G}$ be a continuous homomorphism  with corresponding action of $H$ on $G$ as in \eqref{pre_rho} and canonical map $\gamma\colon{G}\times{H}\to{G}\times{G}$. Then the following are equivalent:
\begin{enumerate}
\item\label{clas1} $\theta$ is a homeomorphism onto its closed image;
\item\label{clas2} the action of $H$ on $G$ is free and proper i.e.~$\gamma$ is injective and proper;
\item\label{clas3} $\bb{\bigl(\rho\bigl(\C_0(G)\bigr)\bigl(\C_0(G)\tens\I_{}\bigr)}=\C_0(G)\tens\C_0(H)$.
\end{enumerate}
\end{lemma}

\begin{proof}
\eqref{clas1}$\Rightarrow$\eqref{clas2}. If $\theta$ is a homeomorphism, then $\gamma$ is a homeomorphism onto its range, as for $(g,g')$ in the range of $\gamma$ (i.e.~$g'=g\cdot{h}$ for some $h\in{H}$) we have $\gamma^{-1}(g,g')=\bigl(g,\theta^{-1}(g^{-1}g')\bigr)$. Hence $\gamma$ is in particular injective and proper.

\eqref{clas2}$\Rightarrow$\eqref{clas1}.
Assume that $\gamma$ is injective and proper. Clearly $\theta$ is then injective. Similarly, if $\theta$ were not proper, then there would be a compact set $K\subset{G}$ with $\theta^{-1}(K)$ non-compact. But then $\gamma^{-1}\bigl(\{e\}\times{K}\bigr)=\{e\}\times\theta^{-1}(K)$ would not be compact either.

Hence $\theta$ is injective and proper. Proper continuous maps between locally compact spaces are automatically closed (\cite[Chapter 1, \S10]{bour}). Hence $\theta$ has a closed image, and as a bijective continuous closed map is in fact a homeomorphism.

\eqref{clas2}$\Leftrightarrow$\eqref{clas3}.
Note first that
\begin{equation*}
\bb{\bigl(\rho\bigl(\C_0(G)\bigr)\bigl(\C_0(G)\tens\I_{}\bigr)}
=\bb{\Gamma\bigl(\C_0(G)\tens\C_0(G)\bigr)}.
\end{equation*}
Hence \eqref{clas3} is equivalent to the fact that $\gamma$ is injective and proper by Theorem \ref{std}.
\end{proof}

\begin{theorem}\label{cl}
Suppose that $G$ and $H$ are locally compact groups. Then the following conditions are equivalent:
\begin{enumerate}
\item\label{cl1} $H$ is a closed quantum subgroup of $G$ in the sense of Vaes;
\item\label{cl2} $H$ is a closed quantum subgroup of $G$ in the sense of Woronowicz;
\item\label{cl3} $H$ is homeomorphic to a closed subgroup of $G$.
\end{enumerate}
\end{theorem}

\begin{proof}
Condition \eqref{cl1} implies \eqref{cl2} by Theorem \ref{Vaes->Wor}. Conditions \eqref{cl2} and \eqref{cl3} are equivalent by Lemma \ref{clas} and Theorem \ref{Woronequiv}. It remains to note that \eqref{cl3} implies \eqref{cl1}. By theorem \ref{VW1} this is precisely \cite[Corollary 4.2.6]{zwarich} (which is a consequence of \cite[Theorem (3.23)]{arsac}).
\end{proof}

\begin{remark}
\noindent
\begin{enumerate}
\item Let $G$ be a locally compact group and let $\HH$ be a locally compact quantum group. If $\HH$ is a closed subgroup of $G$ in the sense of Woronowicz then by Theorem \ref{Woronequiv} there is a surjection from $\C_0(G)$ onto $\C_0(\HH)$, so that $\HH$ is in fact a classical group. By Theorem \ref{cl}, $\HH$ is then also a closed subgroup of $G$ in the usual sense.
\item Let $H$ be a locally compact group and let $\GG$ be a locally compact quantum group. If $H$ is a closed subgroup of $\GG$ in the sense of Woronowicz then the associated morphism $\pi\colon\C_0^\uu(\GG)\to\C_0(H)$ factors through the algebra $\C_0(\tilde{\GG})$, where $\tilde{\GG}$ is the \emph{intrinsic group} of $\GG$ as defined by Kalantar and Neufang (a locally compact group associated to $\GG$, see \cite{kn}). It follows from Theorem \ref{Woronequiv} that $H$ is a closed subgroup of $\tilde{\GG}$ (again in the usual sense).
\end{enumerate}
\end{remark}

Theorem \ref{Vaesfourier} shows that the existence of an injective normal $*$-homomorphism from $\vN(H)$ to $\vN(G)$ is naturally very closely related to the Herz restriction theorem (cf.~Remark \ref{Vf}). To analyze the situation closer assume that $H$ is a closed subgroup of a locally compact group $G$ and consider the following statements:

\begin{enumerate}
\item\label{r1} the restriction map from $\C_0(G)$ to $\C_0(H)$ yields a surjective map from $\cA_G$ to $\cA_H$;
\item\label{r2} the prescription $\lambda_h\mapsto\lambda_h^{(G)}$, $h\in{H}$ (where $\lambda_h$ denotes the (unitary) left shift by $h$ on $L^2(H)$ and $\lambda_h^{(G)}$ the corresponding (unitary) left shift by $h$ on $L^2(G)$) extends to a normal injective $*$-homomorphism from $\vN(H)$ to $\vN(G)$;
\item\label{r3} the restriction of the left regular representation of $G$ to $H$ is quasi-equivalent to the left regular representation of $H$.
\end{enumerate}

It is very easy to see that they are all logically equivalent (\eqref{r2} is essentially the definition of quasi-equivalence in \eqref{r3}, and the equivalence of \eqref{r1} and \eqref{r2} follows from a basic functional analytic argument, appearing already in \cite[Section 0]{h2}). The first condition is the Herz restriction theorem. The third one can be viewed as a statement related to the theory of induced representations, and the induction-restriction procedure, as \cite[Theorem 4.2]{Mackey} states that the left regular representation of $G$ is the induction of the left regular representation of $H$. Interestingly, we could not locate an explicit statement of the condition \eqref{r3} in literature. In the remainder of this section we will give an alternative proof of the implication \eqref{cl3}$\Rightarrow$\eqref{cl1} in Theorem \ref{cl}. In particular this gives a new proof of Herz restriction theorem (cf.~Remark \ref{Vf}). Our reasoning is based on existence of
locally Baire cross-sections for the canonical projection $G\to{G/H}$ (\cite{kehlet}).

Following the notation of \cite[Section 4]{kehlet} we let $q\colon{G/H}\to{G}$ be a locally bounded Baire cross-section to the canonical quotient map $G\to{G/H}$. We denote by $\Phi$ the bijection
\begin{equation*}
(G/H)\times{H}\ni\bigl([g],h\bigr)\longmapsto{q([g])h}\in{G}.
\end{equation*}

Let $\mu$ and $\beta$ be Haar measures on $G$ and $H$ respectively and let $\lambda$ be a quasi-invariant measure on $G/H$ with associated $\rho$-function $\rho$ (\cite[Section 2.6]{foll}). In \cite[Section 4]{kehlet} E.T.~Kehlet shows that
\begin{equation}\label{nazwa}
\psi\longmapsto\rho^{-\frac{1}{2}}\cdot\psi\comp\Phi
\end{equation}
is a unitary map $L^2(G,\mu)\to{L^2\bigl((G/H)\times{H},\lambda\times\beta\bigr)}$. We note that
\begin{equation}\label{hhp}
\Phi\bigl([g],h\bigr)h'=\Phi\bigl([g],hh'\bigr),
\end{equation}
and the function $\rho$ satisfies
\begin{equation}\label{rhorho}
\rho(gh')=\tfrac{\Delta_H(h')}{\Delta_G(h')}\rho(g)
\end{equation}
(\cite{foll}).

We identify $L^2(G/H,\lambda)\tens{L^2(H)}$ with $L^2\bigl((G/H)\times{H},\lambda\times\beta\bigr)$ in the usual way (the respective measures are regular) and define a unitary $T\colon{L^2(G/H,\lambda)}\tens{L^2(H)}\to{L^2(G)}$ as the inverse of \eqref{nazwa}, i.e.
\begin{equation*}
(T\psi)(g)=\rho(g)^{\frac{1}{2}}\psi\bigl(\Phi^{-1}(g)\bigr).
\end{equation*}
For $h'\in{H}$ let $R_{h'}$ be the unitarized operator of right translation by $h'$ on $L^2(H)$ and let $R_{h'}^G$ denote the operator of right translation by $h'$ on $L^2(G)$. Fix $g\in{G}$ and let $\Phi^{-1}(g)=\bigl([g_0],h_0\bigr)$. Taking into account \eqref{hhp} and \eqref{rhorho} we compute
\begin{equation*}
\begin{split}
\Bigl(\bigl(T(\I\tens{R_{h'}})T^*\bigr)\psi\Bigr)(g)
&=\rho(g)^{\frac{1}{2}}
\Bigl(\bigl((\I\tens{R_{h'}})T^*\bigr)\psi\Bigr)\bigl(\Phi^{-1}(g)\bigr)\\
&=\rho(g)^{\frac{1}{2}}\Delta_H(h')^{\frac{1}{2}}(T^*\psi)\bigl([g_0],h_0h'\bigr)\\
&=\rho(g)^{\frac{1}{2}}\Delta_H(h')^{\frac{1}{2}}
\rho\bigl(\Phi\bigl([g_0],h_0h'\bigr)\bigr)^{-\frac{1}{2}}\psi\bigl(\Phi\bigl([g_0],h_0h'\bigr)\bigr)\\
&=\rho(g)^{\frac{1}{2}}\Delta_H(h')^{\frac{1}{2}}
\rho(gh')^{-\frac{1}{2}}\psi(gh')\\
&=\rho(g)^{\frac{1}{2}}\Delta_H(h')^{\frac{1}{2}}
\left(\tfrac{\Delta_H(h')}{\Delta_G(h')}\right)^{-\frac{1}{2}}\rho(g)^{-\frac{1}{2}}\psi(gh')\\
&=\Delta_G(h')^{\frac{1}{2}}\psi(gh')=(R^G_{h'}\psi)(g).
\end{split}
\end{equation*}

Thus $T(\I\tens{R_{h'}})T^*=R^G_{h'}$. This means that the (right) group von Neumann algebra $\vN(H)$ is isomorphic to the von Neumann subalgebra of $\vN(G)$ generated by the right shifts on $G$ by elements from the subgroup $H$. This embedding is the map $\gamma$ from Definition \ref{Vsubgroup}. In particular $H$ is a closed subgroup of $G$ in the sense of Vaes.

\section{Cocommutative case}\label{cocommut}

Let again $G$ and $H$ be locally compact groups. Recall that in this case the dual locally compact quantum groups $\hh{G}$ and $\hh{H}$ of $G$ and $H$ are respectively defined by putting $\C_0(\hh{G})=\cst_r(G)$, $\C_0(\hh{H})=\cst_r(H)$. We have the following result.

\begin{theorem}\label{ccthm}
Let $\pi$ be a morphism from $\hh{H}$ to $\hh{G}$ and let, as usual, $\hh{\pi}$ denote the dual morphism from $G$ to $H$, so that
\begin{equation*}
\hh{\pi}\colon\C_0(H)\ni{f}\longmapsto{f\comp\theta}\in\M\bigl(\C_0(G)\bigr),\qquad(f\in\C_0(H))
\end{equation*}
for some continuous homomorphism $\theta:G \to H$. Then the following conditions are equivalent:
\begin{enumerate}
\item\label{cc:vaes} $\hh{H}$ is a closed quantum subgroup of $\hh{G}$ in the sense of Vaes (via the morphism $\pi$);
\item\label{cc:woro} $\hh{H}$ is a closed quantum subgroup of $\hh{G}$ in the sense of Woronowicz (via the morphism $\pi$);
\item\label{cc:quot} $\theta$ maps $G$ onto $H$ and the induced map $\tilde{\theta}:G/_{\ker{\theta}}\to{H}$ is a homeomorphism.
\end{enumerate}
\end{theorem}

\begin{proof}
That \eqref{cc:vaes}$\Rightarrow$\eqref{cc:woro} is Theorem \ref{Vaes->Wor}.

Suppose that \eqref{cc:woro} holds, so that the morphism
$\pi\colon\cst(G)\to\cst(H)$ maps to $\cst(H)$ and is surjective. Since $G$ and $H$
are classical groups, the algebras $\cA_H$, $\cB_H$, $\cA_G$ and $\cB_G$
(as defined in Subsection \ref{QGsub}) are the classical Fourier and
Fourier-Stieltjes algebras of $H$ and $G$ respectively. In particular we have that
$\cB_H=\cst(H)^*$, and similarly $\cB_G=\cst(G)^*$. Thus
$\pi^*\colon\cB_H\to\cB_G$ is weak$^*$-weak$^*$-continuous. Let $G_0$ be the
closure of the image of $\theta$ in $H$, and let $\theta_0\colon{G}\to{G_0}$ be
the corestriction of $\theta$. By \cite[Lemma~4.2]{is} it follows that $\theta_0$
is an open surjection. We claim that $G_0=H$, from which \eqref{cc:quot} will
follow. Indeed, if $G_0\not=H$ then as $\cA_H$ (and hence also $\cB_H$) is a
\emph{regular} algebra of functions on $H$ (see \cite[Lemme~3.2]{eymard} or
\cite[Proposition~4.1.8]{zwarich}) we can find a non-zero $b\in\cB_H$ with
$b(s)=0$ for all $s\in{G_0}$.
As a map between function algebras, $\pi^*$ is simply $\pi^*(b)=b\comp\theta$, and
so $\pi^*(b)=0$. However, as $\pi$ is surjection, $\pi^*$ is an isometry, and so $\pi^*(b)\not=0$, a contradiction. Thus $G_0=H$ as required.

If \eqref{cc:quot} holds then as both $K=\ker{\theta}$ and $G/K$ are locally compact groups in their own right, they carry Haar measures, which we may normalize so that the Weyl formula holds: for $f\in\C_{00}(G)$,
\begin{equation*}
\int\limits_G{f(s)}\,ds=\int\limits_{G/K}\int\limits_K{f(st)}\,dt\,d(sK).
\end{equation*}
It is not hard to see that the map
\begin{equation*}
I_K\colon{L^1(G)}\ni{f}\longmapsto\int\limits_{K}f(st)\,dt\in{L^1(G/K)}
\end{equation*}
is an algebra homomorphism and a metric surjection; see \cite[Section~1.9.12]{palmer} for example. We notice that then $I_K^*\colon{L^\infty(G/K)}\to{L^\infty(G)}$ is an injective normal $*$-homomorphism which intertwines the coproducts. As $G/K$ is homeomorphic to $H$, the Haar measures on $H$ and on $G/K$ are proportional, and so the map
\begin{equation*}
\gamma_0\colon
L^\infty(H)\ni{F}\longmapsto{F}\comp\tilde{\theta}\in{L^\infty(G/K)}
\end{equation*}
is well-defined, and is hence a normal $*$-isomorphism which intertwines the coproduct. Then set $\gamma=I_K^*\comp\gamma_0\colon{L^\infty(H)}\to{L^\infty(G)}$. So $\gamma$ is an injective normal $*$-homomorphism which intertwines the coproducts, and a simple check shows that $\bigl.\gamma\bigr|_{\C_0(H)}=\hh\pi$, so \eqref{cc:vaes} holds.
\end{proof}

\begin{remark}
Let $G$ be a locally compact group and let $\HH$ be a locally compact quantum group. If $\HH$ is a closed subgroup of $\hh{G}$ in the sense of Woronowicz then $\pi\colon\cst(G)\to\C_0^\uu(\HH)$ is a surjection, and so $\C_0^\uu(\HH)$ is cocommutative, hence of the form $\cst(H)$ for some $H$, and it then follows that $H$ is a quotient of $G$.
\end{remark}

Let us give some indications of how the proof of \cite[Lemma~4.2]{is} proceeds. Firstly, arguing as in the proof of \eqref{cc:quot}$\Rightarrow$\eqref{cc:vaes} above, it is not hard to reduce the problem to the case when $\theta$ is an injection. The key result is then \cite[Theorem~1.3]{bls} which tells us that $\pi^*(\cB_H)$ contains $\cA_G$ (as it is a weak$^*$-closed, conjugate closed, $\cst(G)$-module which, as a space of functions on $G$, separates the points of $G$; this final claim uses the assumption that $\theta$ is injective).

These ideas can be readily generalized to the setting of locally compact quantum groups.

\begin{proposition}\label{ISprop}
Let $\GG$ and $\HH$ be locally compact quantum groups and let $\pi\colon\C_0^\uu(\GG)\to\C_0^\uu(\HH)$ be a strong quantum homomorphism identifying $\HH$ as a Woronowicz-closed subgroup of $\GG$. Furthermore, suppose that $\pi^*(\cB_\HH)$ contains the image of $\cA_\GG$ under the map $\Lambda_{\GG}^*$. Then $\GG$ and $\HH$ are isomorphic.
\end{proposition}

\begin{proof}
As $\pi$ is onto, $\pi^*$ is an isometry onto its range, and so there is an isometric map $\phi\colon{L^\infty(\GG)_*}\to\C_0^\uu(\HH)^*$ with $\pi^*\comp\,\phi=\Lambda_{\GG}^*$; clearly $\phi$ is a Banach algebra homomorphism. Let
\begin{equation*}
\psi=\bigl.\phi^*\bigr|_{\C_0^\uu(\HH)}\colon\C_0^\uu(\HH)\longrightarrow{L^\infty(\GG)}.
\end{equation*}
Then $\psi\comp\pi=\phi^*\comp\bigl.\pi^{**}\bigr|_{\C_0^\uu(\GG)}=\bigl.\Lambda_{\GG}^{**}\bigr|_{\C_0^\uu(\GG)} =\Lambda_{\GG}$, so we have the diagram
\begin{equation*}
\xymatrix{\C_0^\uu(\GG)\ar@{->>}[d]^{\Lambda_{\GG}}\ar@{->>}[r]^\pi&\C_0^\uu(\HH)\ar[ld]^\psi\\\C_0(\GG)}
\end{equation*}
As $\pi$ is onto, it follows that $\psi$ is a $*$-homomorphism. Therefore it follows easily that $\psi$ intertwines the coproducts. From the results of \cite[Section 4]{MRW} there is a strong quantum homomorphism $\psi_0\colon\C_0^\uu(\HH)\to\C_0^\uu(\GG)$ with $\Lambda_{\GG}\comp\psi_0=\psi$. Thus $\Lambda_{\GG}\comp\psi_0\comp\pi=\Lambda_{\GG}$. By passing to bicharacters and applying \cite[Lemma~4.13]{MRW} it follows that $\psi_0\comp\pi$ is the identity on $\C_0^\uu(\GG)$. In particular, $\pi$ must be injective, and so an isomorphism. Thus the quantum groups $\GG$ and $\HH$ are isomorphic by Theorem \ref{isoT}.
\end{proof}

\section{Compact and discrete cases}\label{cptdiscr}

In this section we establish the equivalence of Definitions \ref{Vsubgroup} and \ref{Wsubgroup} when a potential quantum subgroup is compact (Theorem \ref{cpt_sbgrp}) and when the ``larger'' quantum group is discrete (Theorem \ref{discrThm}). The first of these results shows in particular that if both quantum groups in question are compact, the definitions studied in this paper coincide with the one currently adopted in literature (see \cite{BB}, \cite{nchom}, etc.); the second can be thought of as the generalization of the Herz restriction theorem to the context of discrete quantum groups.

\subsection{Compact subgroups}\label{cptsbgrp}

Let $\GG$ and $\HH$ be locally compact quantum groups and assume further that $\HH$ is compact. We will show in Theorem \ref{cpt_sbgrp} that $\HH$ is a closed subgroup of $\GG$ in the sense of Vaes if and only if it is a closed subgroup of $\GG$ in the sense of Woronowicz. However, before proceeding with this theorem let us make the following observation: consider a homomorphism from $\HH$ to $\GG$ described by $\pi\in\Mor\bigl(\C_0^\uu(\GG),\C^\uu(\HH)\bigr)$ (remember that $\HH$ is compact). Based on Theorem \ref{std}\eqref{std4} one could define injectivity of the homomorphism from $\HH$ to $\GG$ as the property that the range of $\pi$ is strictly dense in $\M\bigl(\C^\uu(\HH)\bigr)$. But $\C^\uu(\HH)$ is unital, so strict density of the range of $\pi$ is equivalent to its norm-density. Moreover, since the image of a \cst-algebra under a $*$-homomorphisms is closed, $\pi$ must be a surjection. By Theorem \ref{Woronequiv} this means that $\HH$ is a closed subgroup of $\GG$
in the sense of Woronowicz. In other words, the above argument shows that a compact quantum group $\HH$ with an injective homomorphism into $\GG$ is automatically a closed subgroup of $\GG$ in the sense of Woronowicz (thus by Theorem \ref{cpt_sbgrp} it is also closed in the sense of Vaes). In particular the notion of a quantum subgroup used e.g.~in \cite[Sections 4 and 5]{exqga}, \cite{pa}, \cite{nchom} is identical to those given in Definitions \ref{Vsubgroup} and \ref{Wsubgroup}.

Before proceeding let us also quickly note that a Woronowicz-closed subgroup of a compact quantum group is automatically compact (so, by Theorem \ref{cpt_sbgrp} it is also Vaes-closed). The reason for this is that a quotient of a unital \cst-algebra is obviously unital (cf.~Theorem \ref{Woronequiv}\eqref{Woronequiv4}).

\begin{theorem}\label{cpt_sbgrp}
Let $\HH$ be a closed subgroup of $\GG$ in the sense of Woronowicz and assume that $\HH$ is compact. Then $\HH$ is a closed subgroup of $\GG$ in the sense of Vaes.
\end{theorem}

\begin{proof}
The subgroup $\HH$ is compact, so we can write $\C_0^\uu(\hh{\HH})=\cc_0(\hh{\HH})$, as the quantum group $\hh{\HH}$ is discrete and hence coamenable. Moreover the \cst-algebra $\cc_0(\hh{\HH})$ is a $\cc_0$-direct sum of matrix algebras. It is not difficult to see the following
\begin{itemize}
\item the multiplier algebra $\M\bigl(\cc_0(\hh{\HH})\bigr)$ is canonically isomorphic to the double dual $\cc_0(\hh{\HH})^{**}$,
\item for any \cst-algebra $\sC$ of operators and any $\Phi\in\Mor\bigl(\cc_0(\hh{\HH}),\sC\bigr)$ the extension of $\Phi$ to a mapping $\M\bigl(\cc_0(\hh{\HH})\bigr)\to\M(\sC)\subset{\sC''}$ is $\sigma$-weakly continuous; in fact the extension of $\Phi$ to multipliers coincides with its \emph{normal extension} (\cite[Theorem 3.7.7]{ped}).
\end{itemize}

Now, just as in the proof of Theorem \ref{Woronequiv} (Eq.~\eqref{Lampi}) we have
\begin{equation}\label{Lampi2}
(\Lambda_{\HH}\comp\pi)\bigl((\omega\tens\id)(\wW^{\GG})\bigr)=(\omega\tens\id)(V).
\end{equation}
Since $V\in\M\bigl(\C_0(\GG)\tens\C_0(\HH)\bigr)$ generates $\C_0(\HH)$, we have by Proposition \ref{generatingAU} that
\begin{equation*}
\bigl\{(\eta\tens\id)(V)\St\eta\in{L^\infty(\hh{\GG})_*}\bigr\}
\end{equation*}
is dense in $\C(\HH)$.

Therefore if $\omega\in{L^\infty(\HH)}_*$ is non-zero then it must be non-zero on some element $(\eta\tens\id)(V)$. It follows that
\begin{equation*}
\eta\bigl((\id\tens\omega)(V)\bigr)\neq{0},
\end{equation*}
so $(\id\tens\omega)(V)\neq{0}$. In view of \eqref{Lampi2} this means that $\Lambda_{\hh{\GG}}\comp\hh{\pi}$ is injective on the subspace
\begin{equation*}
\bigl\{(\id\tens\omega)(\Ww^{\HH})\St\omega\in{L^\infty(\HH)_*}\bigr\}\subset\cc_0(\hh{\HH})
\end{equation*}
which coincides with the Fourier algebra
\begin{equation*}
\cA_{\hh{\HH}}=\bigl\{(\id\tens\omega)(\ww^{\HH})\St\omega\in{L^\infty(\HH)_*}\bigr\},
\end{equation*}
as $\hh{\HH}$ is coamenable (cf.~\eqref{forHatH}). This last subspace contains the Pedersen ideal of $\cc_0(\hh{\HH})$ (cf.~\eqref{c00G}). By \cite[Proposition II.8.2.4]{bl} this implies injectivity of $\Lambda_{\hh{\GG}}\comp\hh{\pi}$ on all of $\cc_0(\hh{\HH})$. Finally $\Lambda_{\hh{\GG}}\comp\hh{\pi}$ remains injective after extension to $\ell^\infty(\hh{\HH})=\M\bigl(\cc_0(\hh{\HH})\bigr)$ because this extension coincides with the extension to the multiplier algebra and such extensions always preserve injectivity (\cite[Proposition 2.1]{lance}).
\end{proof}

The arguments similar to these above appeared earlier in \cite{Salmi}, an article which studies the relations between compact quantum subgroups of a coamenable locally compact quantum group $\GG$ and left invariant $\C^*$-subalgebras of $\C_0(\GG)$.

\subsection{Subgroups of discrete quantum groups}

The main result of this subsection is the following:

\begin{theorem}\label{discrThm}
Let $\HH$ be a closed subgroup of $\GG$ in the sense of Woronowicz and assume that $\GG$ is discrete. Then $\HH$ is discrete and $\HH$ is a closed subgroup of $\GG$ in the sense of Vaes.
\end{theorem}

We will prove Theorem \ref{discrThm} by generalizing to the setting of discrete quantum groups the theorem of Herz \cite{herz}, \cite[Proposition 3.23]{arsac} and using Theorem \ref{Vaesfourier} (cf.~Remark \ref{Vf}).

Since $\GG$ is a discrete quantum group, the \cst-algebra $\C_0(\GG)=\C_0^\uu(\GG)=\cc_0(\GG)$ is a $\cc_0$-direct sum:
\begin{equation*}
\cc_0(\GG)=\bigoplus_{\alpha\in\mathcal{R}}M_{n_\alpha}
\end{equation*}
and the embedding of $\cB_\GG$ into $\M\bigl(\C_0^\uu(\GG)\bigr)$ is in this case
\begin{equation}\label{mapping}
\cB_\GG=\C^\uu(\hh{\GG})^*\ni\eta\longmapsto
(\eta\tens\id)(\WW^{\GG})\in\M\bigl(\cc_0(\GG)\bigr)=\ell^\infty(\GG)
\end{equation}
(and $\cA_\GG$ is then mapped to the space of slices of $\ww^\GG=\wW^\GG$ with normal functionals on $L^\infty(\hh{\GG})$). In particular one can use the functionals dual to the canonical basis of the Hopf $*$-algebra sitting inside $\C(\GG)$ (\cite[Theorem 2.2]{cqg}, \cite[Theorem 5.1]{bmt}). These are normal and we easily see that their image in the mapping \eqref{mapping} spans the Pedersen ideal $\cc_{00}(\GG)$ of $\cc_0(\GG)$. The ideal $\cc_{00}(\GG)$ is the algebraic direct sum of the same family of matrix algebras. On the other hand these functionals are linearly dense in $L^\infty(\hh{\GG})_*$ (they correspond to density matrices on $L^2(\hh{\GG})$ which are of finite rank). Therefore
\begin{equation}\label{c00G}
\cc_{00}(\GG)\subset\cA_\GG
\end{equation}
with $\cA_\GG$ viewed as a subspace of $\cc_0(\GG)$.

\begin{theorem}\label{qHerz}
The space $\cA_\GG$ is the closure in $\cB_\GG$ of $\cc_{00}(\GG)$.
\end{theorem}

\begin{proof}
As we mentioned before stating Theorem \ref{qHerz} the space $\cc_{00}(\GG)$ viewed inside $\cB_\GG$ is the space of functionals which are normal on $L^\infty(\hh{\GG})$ and whose density matrix is a finite rank operator. The closure of this space of functionals inside the space of all functionals on $\C^\uu(\hh{\GG})$ is the space of all functionals which are normal on $L^\infty(\hh{\GG})$, i.e.~the space $\cA_\GG$.
\end{proof}

\begin{proof}[Proof of Theorem \ref{discrThm}]
The \cst-algebra $\C_0(\GG)=\cc_0(\GG)$ is a $\cc_0$-direct sum of matrix algebras:
\begin{equation*}
\cc_0(\GG)=\bigoplus_{\alpha\in\mathcal{R}}M_{n_\alpha}.
\end{equation*}
By \cite[Proposition II.8.2.4]{bl} any ideal in $\cc_0(\GG)$ is of the form
\begin{equation*}
\bigoplus_{\alpha\in\mathcal{R}_0}M_{n_\alpha}.
\end{equation*}
for some $\mathcal{R}_0\subset\mathcal{R}$ (the direct sum is still in $\cc_0$-sense). Now if $\pi\colon\cc_0(\GG)\to\C_0^\uu(\HH)$ is the epimorphism corresponding to the embedding of $\HH$ into $\GG$ and $\mathcal{R}_0$ corresponds to the kernel of $\pi$, we see that $\C_0^\uu(\HH)$ is the $\cc_0$-direct sum
\begin{equation*}
\C_0^\uu(\GG)=\bigoplus_{\alpha\in\mathcal{R}\setminus\mathcal{R}_0}M_{n_\alpha}.
\end{equation*}
For the same reason the algebra $\C_0(\HH)$ which is a (potentially proper) quotient of $\C_0^\uu(\HH)$ is also a $\cc_0$-direct sum of matrix algebras. In particular $\C_0(\HH)$ is an ideal in $\C_0(\GG)^{**}$. By \cite[Theorem 4.4]{runde} $\HH$ is a discrete quantum group. In particular $\C_0^\uu(\HH)=\C_0(\HH)=\cc_0(\HH)$.

Consider the adjoint of the map $\hh{\pi}\colon\C^\uu(\hh{\HH})\to\C^\uu(\hh{\GG})$, i.e.
\begin{equation*}
{\hh{\pi}}^*\colon\cB_\GG\longrightarrow\cB_{\HH}.
\end{equation*}
We now note that ${\hh{\pi}}^*$ maps $\cc_{00}(\GG)$ into $\cc_{00}(\HH)$. Indeed, ${\hh{\pi}}^*$ is the operation of pre-composing a functional with $\hh{\pi}$. In particular, on the level of $\ell^\infty(\GG)$, where $\cB_{\GG}$ is embedded, we have
\begin{equation*}
\begin{split}
{\hh{\pi}}^*\bigl((\eta\tens\id)(\WW^\HH)\bigr)
&=\bigl([\eta\comp\hh{\pi}]\tens\id\bigr)(\WW^\HH)\\
&=(\eta\tens\id)\bigl((\hh{\pi}\tens\id)(\WW^\HH)\big)\\
&=(\eta\tens\id)\bigl((\id\tens\pi)(\WW^\GG)\big)\\
&=\bar{\pi}\bigl((\eta\tens\id)(\WW^\GG)\bigr),
\end{split}
\end{equation*}
where $\bar{\pi}$ is the canonical extension of $\pi$ to $\M\bigl(\cc_0(\GG)\bigr)=\ell^\infty(\GG)$. Also ${\hh{\pi}}^*$ is a contraction for the norms on $\cB_\GG$ and $\cB_\HH$ (as an adjoint map of a contraction $\hh{\pi}\colon\C^\uu(\hh{\HH})\to\C^\uu(\hh{\GG})$). It follows from Theorem \ref{qHerz} that ${\hh{\pi}}^*$ restricts to a contraction
\begin{equation*}
T\colon\cA_{\GG}\longrightarrow\cA_\HH
\end{equation*}
with dense range; this completes the proof by applying Theorem~\ref{Vaesfourier}.
\end{proof}

\section*{Acknowledgments}

The authors wish to thank Nico Spronk for helpful suggestions and pointing out several important references. The second and third authors were supported by the National Science Centre (NCN) grant no.~UMO-2011/01/B/ST1/06474. The fourth author was supported by National Science Centre (NCN) grant no.~2011/01/B/ST1/05011. The third author's visit to the University of Leeds was supported by EPSRC grant EP/I002316/1.


\begin{thebibliography}{W1995}
\bibitem{arsac}
G.~Arsac: Sur l'espace de Banach engendr\'e par les coefficients d'une repr\'esentation unitare. \emph{Pub.~D\`ep. Math.~Lyon} \textbf{13} (1976), 1--101.
\bibitem{bs}
S.~Baaj \& G.~Skandalis: Unitaires multiplicatifs et dualit\'e pour les produits crois\'es de $\mathrm{C}^*$-alg\`ebres. \emph{Ann.~Scient.~\'Ec.~Norm.~Sup.}, $4^{\text{\tiny e}}$ s\'erie, t.~\textbf{26} (1993), 425--488.
\bibitem{BB}
T.~Banica \& J.~Bichon: Quantum groups acting on 4 points. \emph{J.~Reine Angew.~Math.} \textbf{626} (2009), 75--114.
\bibitem{nchom}
T.~Banica, A.~Skalski \& P.M.~So{\l}tan: Noncommutative homogeneous spaces: the matrix case. \emph{J.~Geom.~Phys.} \textbf{62} (2012), 1451--1466.
\bibitem{bmt}
E.~Bedos, G.J.~Murphy \& L.~Tuset: Co-amenability for compact quantum groups. \emph{J.~Geom.~Phys.} \textbf{40} (2001), 130--153.
\bibitem{bls}
M.E.B.~Bekka, A.T.~Lau \& G.~Schlichting: On invariant subalgebras of the Fourier-Stieltjes algebra of a locally compact group. \emph{Math.~Ann.} \textbf{294} (1992), 513--522.
\bibitem{bl}
B.~Blackadar: \emph{Operator algebras. Theory of $\mathrm{C}^*$-algebras and von Neumann algebras.} Encyclopedia of Mathematical Sciences, Vol.~\textbf{122}, Springer-Verlag 2006.
\bibitem{bour}
N.~Bourbaki: \emph{Elements of Mathematics. General Topology. Chapters 1-4.} Springer-Verlag 1989.
\bibitem{disser}
M.~Daws: Multipliers, Self-Induced and Dual Banach Algebras. \emph{Dissertationes Mathematicae (Rozprawy Matematyczne)} \textbf{470} (2010).
\bibitem{DelDer} J.~Delaporte \& A.~Derighetti:
On Herz' extension theorem. \emph{Boll.~Un.~Mat.~Ital.~A}  \textbf{6} (1992), 245–-247.
\bibitem{Der} A.~Derighetti:  Relations entre les convoluteurs d'un groupe localement compact et ceux d'un sous-groupe ferm\'e. \emph{Bull.~Sci.~Math.} \textbf{106} (1982), 69–-84.
\bibitem{eymard}
P.~Eymard: L'alg\`ebre de Fourier d'un groupe localement compact. \emph{Bull.~Soc.~Math.~France} \textbf{92}  (1964), 181--286.
\bibitem{foll}
G.B.~Folland: \emph{A course in abstract harmonic analysis}. CRC Press 1995.
\bibitem{FST}
U.~Franz, A.~Skalski \& R.~Tomatsu: Idempotent states on the compact quantum groups and their classification on $U_q(2)$, $SU_q(2)$ and $SO_q(3)$. To appear in \emph{Journal of Noncommutative Geometry}, available at \texttt{arXiv:0903.2363v2 [math.OA]}.
\bibitem{herz}
C.~Herz: Le rapport entre l'alg\`ebre $A_p$ d'un groupe et d'un sous-groupe. \emph{C.~R.~Acad.~Sci.~Paris S\'er.~A-B} \textbf{271} (1970), 244--246.
\bibitem{h2}
C.~Herz: Harmonic synthesis for subgroups. \emph{Ann.~Inst.~Fourier (Grenoble)} \textbf{23} (1973), 91--123.
\bibitem{hnr}
Z.~Hu, M.~Neufang \& Z.-J.~Ruan: Completely bounded multipliers over locally compact quantum groups. \emph{Proc.~Lond.~Math.~Soc.~(3)} \textbf{103} (2011), 1--39.
\bibitem{is}
M.~Ilie \& R.~Stokke: Weak$^*$-continuous homomorphisms of Fourier-Stieltjes algebras, \emph{Math.~Proc.~Cambridge Philos.~Soc.} \textbf{145} (2008), 107--120.
\bibitem{kn}
M.~Kalantar \& M.~Neufang: From Quantum Groups to Groups. Available at \texttt{arXiv:1110.5129v1 [math.OA]}.
\bibitem{kasp}
P.~Kasprzak: Rieffel deformation of homogeneous spaces. \emph{J.~Funct.~Anal.} \textbf{260} (2011), 146--163.
\bibitem{kehlet}
E.T.~Kehlet: Cross sections for quotient maps of locally compact groups.\emph{Math.~Scand.} \textbf{55} (1984), 152--160.
\bibitem{Johanuniv}
J.~Kustermans: Locally compact quantum groups in the universal setting. \emph{Int.~J.~Math.} \textbf{12} (2001) 289--338.
\bibitem{KV}
J.~Kustermans \& S.~Vaes: Locally compact quantum groups. \emph{Ann.~Scient.~\'{E}c.~Norm.~Sup.} $4^{\text{\tiny e}}$ s\'{e}rie, t.~\textbf{33} (2000), 837--934.
\bibitem{lance}
C.E.~Lance: \emph{Hilbert $\mathrm{C}^*$-modules: a toolkit for operator algebraists.} London Mathematical Society Lecture Note Series \textbf{210}, Cambridge University Press 1995.
\bibitem{mnw}
T.~Masuda, Y.~Nakagami \& S.L.~Woronowicz: A $\mathrm{C}^*$-algebraic framework for the quantum groups. \emph{Int.~J.~Math.} \textbf{14} (2003), 903--1001.
\bibitem{Mackey}
G.~Mackey: Induced representations of locally compact groups I. \emph{Ann.~Math.} \textbf{55} (1952), 101--139.
\bibitem{MRW}
R.~Meyer, S.~Roy \& S.L.~Woronowicz: Homomorphisms of quantum groups, \emph{M\"unster J.~Math.} \textbf{4} (2011), 101--124.
\bibitem{palmer}
T.W.~Palmer: \emph{Banach algebras and the general theory of $*$-algebras, Volume 1,} Cambridge University Press 1994.
\bibitem{ped}
G.K.~Pedersen: \emph{$\mathrm{C}^*$-algebras and their automorphism groups.} Academic Press 1979.
\bibitem{Pin}
C.~Pinzari: Embedding ergodic actions of compact quantum groups on $\mathrm{C}^*$-algebras into quotient spaces. \emph{Int.~J.~Math.} \textbf{18} (2007), 137--164.
\bibitem{podlesPhD}
P.~Podle\'s: \emph{Przestrzenie kwantowe i ich grupy symetrii} (\emph{Quantum spaces and their symmetry groups}). Ph.D.~Thesis, Department of Mathematical Methods in Physics, Faculty of Physics, University of Warsaw (1989) (in Polish).
\bibitem{podles}
P.~Podle\'s: Symmetries of quantum spaces. Subgroups and quotient spaces of quantum $\mathrm{SU}(2)$ and $\mathrm{SO}(3)$ groups. \emph{Commun.~Math.~Phys.} \textbf{170} (1995), 1--20.
\bibitem{pw}
P.~Podle\'s \& S.L.~Woronowicz: Quantum deformation of Lorentz group. \emph{Comm.~Math.~Phys.} \textbf{130} (1990), 381--431.
\bibitem{runde}
V.~Runde: Characterizations of compact and discrete quantum groups through second duals. \emph{J.~Op.~Th.} \textbf{60} (2008), 415--428.
\bibitem{Salmi} 
P.~Salmi: Compact quantum subgroups and left invariant $\mathrm{C}^*$-subalgebras of locally compact quantum groups.
   \emph{J.~Funct.~Anal.} \textbf {261} (2011), 1--24.
\bibitem{pa}
P.~Salmi \& A.~Skalski: Idempotent states on locally compact quantum groups. To appear in \emph{Quarterly Journal of Mathematics}, available at \texttt{arXiv:1102.2051v2 [math.OA]}.
\bibitem{schneider}
H.-J.~Schneider: Principal homogeneous spaces for arbitrary Hopf algebras. \emph{Isr.~J.~Math.} \textbf{72} (1990), 167--195.
\bibitem{exqga}
P.M.~So{\l}tan: Examples of non-compact quantum group actions. \emph{J.~Math.~Anal.~Appl.} \textbf{372} (2010), 224--236.
\bibitem{modmu}
P.M.~So{\l}tan \& S.L.~Woronowicz: A remark on manageable multiplicative unitaries. \emph{Lett.~Math.~Phys.} \textbf{57} (2001), 239--252.
\bibitem{WorSol}
P.M.~So{\l}tan \& S.L.~Woronowicz: From multiplicative unitaries to quantum groups II. \emph{J.~Funct.~Anal.} \textbf{252} (2007), 42--67.
\bibitem{Tak}
M.~Takesaki: \emph{Theory of Operator Algebras I.} Springer-Verlag 1979.
\bibitem{TT}
M.~Takesaki \& N.~Tatsuuma: Duality and subgroups. \emph{Ann.~Math.} \textbf{93} (1971), 344--364.
\bibitem{Vaes}
S.~Vaes: A new approach to quantum and imprimitivity results. \emph{J.~Funct.~Anal.} \textbf{229} (2005), 317--374.
\bibitem{VV}
S.~Vaes \& L.~Vainerman: On low-dimensional locally compact quantum groups. In \emph{Locally compact quantum groups and groupoids, IRMA Lect.~Math.~Theor.~Phys.} \textbf{2}, de Gruyter, Berlin, (2003), pp.~127--187.
\bibitem{VV2}
S.~Vaes \& L.~Vainerman: Extensions of locally compact quantum groups and the bicrossed product construction. \emph{Adv.~Math.} \textbf{175} (2003), 1--101.
\bibitem{W-O}
N.E.~Wegge-Olsen: \emph{$\mathrm{K}$-theory and $\mathrm{C}^*$-algebras: a friendly approach}. Oxford University Press 1993.
\bibitem{pseu}
S.L.~Woronowicz: Pseudogroups, pseudospaces and Pontryagin duality. \emph{Proceedings of the International Conference on Mathematical Physics, Lausanne 1979} Lecture Notes in Physics, \textbf{116}, pp.~407--412.
\bibitem{unbo}
S.L.~Woronowicz: Unbounded elements affiliated with $\mathrm{C}^*$-algebras and non-compact quantum groups. \emph{Commun.~Math.~Phys.} \textbf{136} (1991), 399--432.
\bibitem{gen}
S.L.~Woronowicz: $\mathrm{C}^*$-algebras generated by unbounded elements. \emph{Rev.~Math.~Phys.} \textbf{7}, (1995), 481--521.
\bibitem{mu}
S.L.~Woronowicz: From multiplicative unitaries to quantum groups. \emph{Int.~J.~Math.} \textbf{7}, (1996), 127--149.
\bibitem{cqg}
S.L.~Woronowicz: Compact quantum groups. In: \emph{Sym\'etries quantiques, les Houches, Session LXIV 1995,} Elsevier 1998, pp.~845--884.
\bibitem{zwarich}
C.~Zwarich: \emph{Von Neumann algebras for abstract harmonic analysis}. Thesis at University of Waterloo, available at \texttt{http://hdl.handle.net/10012/3920}.
\end{thebibliography}
\end{document}